\documentclass{amsart}
\usepackage{amsmath,amsthm,amssymb,amsfonts,amsmath,amscd}
\usepackage[francais]{babel}
\title[Exposants dans les paquets d'Arthur]{Comparaisons des exposants \`a l'int\'erieur d'un paquet d'Arthur archim\'edien \\
Exponents in Archimedean Arthur packets}

\author{Nicolas Bergeron et Laurent Clozel} 
\thanks{N.B. \& L.C sont membres de l'Institut Universitaire de France. Pendant la r\'edaction de cet article L.C \'etait partiellement financ\'e par la Florence Gould Foundation, l'Ellentuck Fund et le James D. Wolfensohn Fund.}
\address{Institut de Math\'ematiques de Jussieu \\
Unit\'e Mixte de Recherche 7586 du CNRS \\
Universit\'e Pierre et Marie Curie \\
4, place Jussieu 75252 Paris Cedex 05, France \\}
\email{bergeron@math.jussieu.fr}
\urladdr{http://people.math.jussieu.fr/~bergeron}

\address{Universit\'e Paris Sud \\
Unit\'e Mixte de Recherche 8628 du CNRS \\
Laboratoire de Math\'ematiques \\
B\^at. 425, 91405 Orsay cedex, France\\}
\email{Laurent.Clozel@math.u-psud.fr}


 \DeclareFontFamily{OT1}{rsfs}{}
 
\DeclareFontShape{OT1}{rsfs}{n}{it}{<-> rsfs10}{}
\DeclareMathAlphabet{\mathscr}{OT1}{rsfs}{n}{it}

\newcommand{\C}{\mathbb{C}}

\newcommand{\Z}{\mathbf{Z}}

\newcommand{\B}{\mathbf{B}}

\newcommand{\T}{\mathbf{T}}

\renewcommand{\a}{\mathfrak{a}}

\newcommand{\n}{\mathfrak{n}}

\DeclareFontFamily{OT1}{rsfs}{}

\DeclareFontShape{OT1}{rsfs}{n}{it}{<-> rsfs10}{}
\DeclareMathAlphabet{\mathscr}{OT1}{rsfs}{n}{it}

\newcommand{\G}{\mathbf{G}}
\newcommand{\LL}{\mathbf{L}}

\newcommand{\R}{\mathbf{R}}
\newcommand{\SO}{\mathrm{SO}}

\newcommand{\TT}{\mathbf{T}}

\newcommand{\g}{\mathfrak{g}}

\newcommand{\uu}{\mathfrak{u}}

\swapnumbers
\newtheorem{thm}[subsection]{Th\'eor\`eme}  
\newtheorem{lem}[subsection]{Lemme}         
\newtheorem*{lem*}{Lemme}         
\newtheorem{prop}[subsection]{Proposition}
\newtheorem*{prop*}{Proposition}

\theoremstyle{definition}

\numberwithin{equation}{subsection}

\newcommand{\GL}{\mathrm{GL}}
\newcommand{\SL}{\mathrm{SL}}

\newcommand{\Sp}{\mathrm{Sp}}

\newcommand{\N}{\mathbf N}
\newcommand{\RR}{\mathbf R}
\newcommand{\CC}{\mathbf C}
\newcommand{\ZZ}{\mathbf Z}

\def\adots{\mathinner{\mkern2mu\raise1pt\hbox{.}
\mkern3mu\raise4pt\hbox{.}\mkern1mu\raise7pt\hbox{.}}}

\begin{document}

\begin{abstract}  
En g\'en\'eralisant la d\'emonstration de Hecht et Schmid de la conjecture d'Osborne nous d\'emontrons une version archim\'edienne -- et plus faible -- d'un th\'eor\`eme de Colette Moeglin. 
Cela donne un sens archim\'edien pr\'ecis au principe \'enonc\'e par le second auteur selon
lequel {\it on trouve dans un paquet d'Arthur des repr\'esentations qui appartiennent au paquet de 
Langlands associ\'e et des repr\'esentations plus temp\'er\'ees.} 

Generalizing the proof -- by Hecht and Schmid -- of Osborne's conjecture we prove an Archimedean (and weaker) version of a theorem of Colette Moeglin. The result we obtain is a precise Archimedean version
of the general principle -- stated by the second author -- according to which {\it a local Arthur packet contains the corresponding local $L$-packet and representations which are more tempered.} 

\end{abstract}
\maketitle
\tableofcontents

\section{Introduction}

Dans cet article nous formulons et d\'emontrons l'analogue de la formule d'Osborne \cite[Thm. 3.6]{HechtSchmid} dans le cas des groupes tordus. Cela permet de r\'eexprimer les identit\'es de caract\`eres
consid\'er\'ees par Arthur dans \cite[Thm. 30.1]{Arthur} en termes d'identit\'es de traces sur la $\mathfrak{n}$-homologie. En guise d'application, nous comparons les exposants des repr\'esentations des groupes intervenant dans un m\^eme paquet d'Arthur sur $\CC$.

Consid\'erons en effet un param\`etre d'Arthur
$$\psi = \chi_1 \otimes R_{a_1} \oplus \ldots \oplus \chi_m \otimes R_{a_m}$$ 
d'image contenue dans $\widehat{G} \subset \GL (N , \CC)$,
o\`u chaque $\chi_j$ est un caract\`ere unitaire de $\CC^*$ que l'on \'ecrit 
$z \mapsto z^{p_j} \bar z ^{q_j}$ avec $\mathrm{Re} (p_j + q_j) = 0$, chaque $R_{a_j}$ est 
une repr\'esentation de dimension $a_i$ de $\SL(2 , \CC)$ et $\widehat{G} = \SO (N , \CC)$
ou $\Sp (N,\CC)$.  
On associe au param\`etre $\psi$ la 
repr\'esentation de $\GL(N , \CC)$~: 
\begin{equation} \label{Pi}
\Pi = \Pi_{\psi} =  \mathrm{ind} (  \chi_1 \circ \det \otimes \ldots \otimes \chi_m \circ \det )
\end{equation}
(induction unitaire \`a partir du parabolique de type $(a_1 , \ldots , a_m)$). 
Rappelons qu'Arthur associe \`a $\psi$ un paquet fini $\small\prod (\psi)$ de repr\'esentations du groupe
complexe $G$ de groupe dual $\widehat{G}$, voir \cite[Thm. 30.1]{Arthur}. 

Comme nous le rappelons au \S \ref{S2}, il correspond naturellement au groupe $\GL (N, \CC)$ {\it tordu} (par un automorphisme involutif $\theta$) un syst\`eme de racines r\'eduit de groupe associ\'e $\SO (2\ell +1 , \CC)$ o\`u 
$\ell = [N/2]$. Le tore diagonal de $G$ est naturellement isomorphe au tore diagonal de $\SO (2\ell +1 , \CC)$; notons $A\cong \RR^{\ell}$ sa partie d\'eploy\'ee et $\mathfrak{a}$ son alg\`ebre de Lie r\'eelle. Un {\it exposant} de $G$
est un \'el\'ement de $\mathrm{Hom} (\mathfrak{a} , \CC ) = \CC^{\ell}$. On note
$e_{\psi}$ l'exposant $(m_1 \leq \ldots \leq m_{\ell}) \in \CC^{\ell}$ o\`u les $m_j \in \N$ sont 
les \'el\'ements des segments $\sigma_i = \left( a_i -1, a_i -2 , \ldots  \right)$, rang\'es par ordre croissant. 

Il correspond au groupe $\SO (2\ell +1 , \CC)$ un ordre sur les exposants donn\'e par les racines de
$\SO (2\ell +1 , \CC)$~:
$$e \leq_{\theta} e' \Leftrightarrow e' = e + \sum_{\alpha} n_{\alpha} \alpha \quad (n_{\alpha} \geq 0)$$
o\`u $\alpha$ d\'ecrit les racines simples de $\SO (2\ell +1 , \CC)$. Noter que les racines de $G$ 
d\'efinissent aussi un ordre sur $\CC^{\ell}$; si $e$ est inf\'erieur \`a $e'$ pour l'ordre d\'efini par les racines de $G$ alors $e\leq_{\theta} e'$, la r\'eciproque n'est pas vraie en g\'en\'eral. Cela \'etant, on a le r\'esultat suivant qui est une version archim\'edienne, et plus faible, d'un th\'eor\`eme de Colette Moeglin
\cite{Moeglin}.

\begin{thm} \label{thm:intro}
Soit $\pi$ une repr\'esentation arbitraire de $\small\prod (\psi)$ et $e$ un exposant minimal de $\pi$.
Alors~:
$$\mathrm{Re} (e) \geq_{\theta} e_{\psi}.$$
\end{thm}

Nous avons adopt\'e, sur les caract\`eres r\'eels de tores d\'eploy\'es, l'ordre usuel dans ces questions, voir \cite{HechtSchmid}. Le th\'eor\`eme dit donc que les coefficients de $\pi$ {\it d\'ecroissent plus vite}
que le caract\`ere associ\'e \`a $e_{\psi}$, selon le principe g\'en\'eral \'enonc\'e dans \cite{ClozelABS}.

\medskip

{\it Nous remercions le rapporteur, dont la lecture nous a permis de corriger une erreur caract\'eris\'ee dans la version originale.}

\section{Rappels sur les groupes tordus} \label{S2}

\subsection{} \label{2.1} Soit $\G$ un groupe r\'eductif et connexe sur $\R$ et $\theta$ un automorphisme
de $\G$ d'ordre fini \'egal \`a $d$ et d\'efini sur $\R$. On identifie $\G$ au groupe de ses
points complexes et on note $\G_{\rm sc}$ le rev\^etement simplement connexe de son groupe d\'eriv\'e. On a 
une application naturelle 
$${\rm Aut} (\G) \rightarrow {\rm Aut} (\G_{\rm sc}).$$

Consid\'erons l'espace alg\'ebrique tordu (voir \cite{Labesse})~:
$$\LL = \G \rtimes \theta  \subset \G \rtimes {\rm Aut} (\G ).$$
\'Etant donn\'e un \'el\'ement $\delta$ dans $\LL$ on note $\G^{\delta}$ le groupe des points fixes 
de l'automorphisme ${\rm Ad}_{\LL} (\delta )$. C'est le centralisateur de $\delta$. Nous adoptons la notation
d'Arthur $\G_{\delta}$ pour d\'esigner la composante connexe de l'\'el\'ement neutre dans $\G^{\delta}$. 

\medskip
\noindent
{\it Exemple.} Pour les applications que nous consid\`ererons on peut prendre pour $\G$ le groupe 
des automorphismes lin\'eaires de $F^N$ avec $F= \R$ ou $\CC$ et pour $\theta$ l'automorphisme involutif 
op\'erant par $g \mapsto g^{\theta}=J {}^t \- g^{-1} J^{-1}$ avec 
$$J = \left(
\begin{array}{cccc}
& & & - 1 \\
& & 1 & \\
& \adots & & \\
(-1)^N & & & 
\end{array} \right).$$ 
Il correspond \`a $J$ une forme bilin\'eaire non d\'eg\'en\'er\'ee sur $F^N$. Le
groupe $\G$ agit naturellement \`a gauche et \`a droite sur l'ensemble des formes bilin\'eaires non d\'eg\'en\'er\'ees sur $F^N$. En particulier~: pour $g$ et $g'$ dans $\G$ on a $(g,g') \cdot J = g J {}^t \- (g')^{-1}$. On peut alors identifier $\LL$ \`a l'espace alg\'ebrique des formes bilin\'eaires non d\'eg\'en\'er\'ees sur $F^N$ par l'application $(g , \theta) \mapsto   g J$.
\medskip

Soit $\B$ un sous-groupe
de Borel dans $\G$ et $\T$ un tore maximal. Suivant la terminologie de 
Kottwitz et Shelstad \cite{KS} on appelle la donn\'ee $(\B,\T)$ une {\it paire} dans $\G$.
Rappelons qu'un \'el\'ement $\delta$ est {\it semisimple} si ${\rm Ad}_{\LL} (\delta )$ est semisimple.
Un \'el\'ement semisimple {\it r\'egulier} est un \'el\'ement $\delta$ dans $\LL$ dont le centralisateur connexe $\G_{\delta}$ est un tore. 

\begin{lem} \cite[Lem. II.1.1]{Labesse}
Soit $\delta \in \LL$ semisimple et soit $(\B, \T)$ une paire $\delta$-stable. Alors
$\T_{\delta} = \T \cap \G_{\delta}$ et $\B_{\delta} = \B \cap \G_{\delta}$ d\'efinissent une paire dans $\G_{\delta}$. R\'eciproquement, si $(\B_{\delta} , \T_{\delta})$ est une paire dans $\G_{\delta}$, le centralisateur $\T= {\rm Cent} (\T_{\delta} , \G)$ est le tore maximal d'une paire $(\B,\T)$ dans $\G$ telle que
$\B\cap \G_{\delta} = \B_{\delta}$.
\end{lem}

Un {\it \'epinglage} de $\G$ est un triplet 
$(\B,\T,\{X\})$, o\`u $(\B,\T)$ est une paire dans $\G$ et $\{X\}$ un ensemble de vecteurs propres, un pour 
chaque racine simple de $\T$ dans $\B$. Nous supposerons dor\'enavant que l'automorphisme 
$\theta$ pr\'eserve un \'epinglage $(\B,\T,\{X\})$. Notons\footnote{Noter que vu les notations introduites 
en \ref{2.1} $\T_{\theta}$ d\'esigne (essentiellement) les invariants, et non les coinvariants, de $\theta$.} $\T^1 = \T_{\theta}$ et $\G^1 = \G_{\theta}$.

\begin{lem} \cite[Lem. II.1.2]{Labesse}
Le tore $\T^1$ est maximal dans $\G^1$, $\T^1 = \G^1 \cap \T^{\theta}$ et $\T = {\rm Cent} (\T^1 , \G)$.
Enfin, il existe des isomorphismes canoniques entre les groupes de Weyl~:
$$W(\G^1 , \T^1 ) \stackrel{\sim}{\rightarrow} W (\G , \T)^{\theta}.$$
Un \'el\'ement $w \in W(\G , \T)$ appartient \`a $W (\G , \T)^{\theta}$ si et seulement s'il laisse
stable $\T^{\theta}$.
\end{lem}

Dans la suite nous notons $W$ le groupe $W (\G , \T)^{\theta}$; c'est naturellement un sous-groupe 
de $W (\G , \T)$ il est isomorphe au groupe de Weyl $W (\G^{\theta} , \T^{\theta})$  
(pour des groupes non connexes, voir Kottwitz-Shelstad
\cite{KS} pour les d\'efinitions).

Soit $R(\B,\T)$ l'ensemble des racines de $\T$ dans $\B$. Alors $R(\B^1 , \T^1)$ est contenu dans
$$R_{\rm res} = \{\alpha_{\rm res} = \alpha_{|\T^1} \; : \; \alpha \in R(\B,\T) \}.$$
De plus, $\alpha_{\rm res}$ appartient \`a $R(\B^1 , \T^1)$ si et seulement si $\theta$ fixe un vecteur
de l'espace engendr\'e par les espaces de racines $\beta \in R(\B,\T)$ telles que $\beta_{\rm res} = 
\alpha_{\rm res}$. Kottwitz et Shelstad distinguent trois types de racines restreintes $\alpha_{\rm res}$~:
\begin{itemize}
\item Type $R_1$ :  $2\alpha_{\rm res}$, $\frac12 \alpha_{\rm res} \notin R_{\rm res}$.
\item Type $R_2$ :  $2\alpha_{\rm res}\in R_{\rm res}$.
\item Type $R_3$ :  $\frac12 \alpha_{\rm res} \in R_{\rm res}$.
\end{itemize}

\medskip
\noindent
{\it Exemple.} (Voir Waldspurger \cite{Waldspurger2}.) Si $\G$ est le groupe des automorphismes lin\'eaires de $F^N$, 
on peut prendre pour $(\B, \T)$ sa paire usuelle. Alors, 
l'automorphisme involutif $\theta$ op\'erant par $g \mapsto g^{\theta}=J {}^t \- g^{-1} J^{-1}$ pr\'eserve la paire $(\B , \T)$ et son \'epinglage standard. Les racines de $(\T , \G)$ sont
donn\'ees par 
$$\alpha_{i,j} : (z_1 , \ldots , z_N ) \mapsto z_i / z_j  \ \ \ (i \neq j)$$
et les racines dans $\B$ correspondent \`a la base $(x_1 -x_2 , 
\ldots , x_{N-1} - x_N )$.
L'automorphisme $\theta$ envoie $\alpha_{i,j}$ sur $\alpha_{N+1-j , N+1-i}$.

Un \'el\'ement $z \in \T$ appartient \`a $\T^1$ si et seulement si $z_{(N+1)/2} = 1$ (quand
$N$ est impair) et $z_{N+1-i} = z_{i}^{-1}$ pour $i \neq (N+1)/2$. On note $\ell = \left[N/2\right]$.  
Rappelons que l'on a identifi\'e $J$ \`a une forme bilin\'eaire non d\'eg\'en\'er\'ee sur $F^N$; 
c'est une forme symplectique si $N$ est pair, quadratique si $N$ est impair. Le groupe $\G^1$ est 
alors la composante neutre de son groupe d'automorphismes; c'est un groupe quasi-d\'eploy\'e dont 
$\T^1$ est un sous-tore maximal. Le groupe de Weyl $W$ est en particulier identifi\'e \`a 
$\mathfrak{S}_{\ell} \rtimes \{ \pm 1 \}^{\ell}$.

Enfin, si $\alpha = \alpha_{i,j} \in R(\B,\T)$ on a 
$\alpha_{\rm res} \in R(\B^1 , \T^1)$ si $i \neq N+1-j$ et $\frac12 \alpha_{\rm res} \in R(\B^1 , \T^1)$ sinon. Et $\alpha_{\rm res}$ est de type $R_1$ si $i \neq N+1-j$ et $i$ et $j$ sont tous les deux diff\'erents de $(N+1)/2$, de type $R_2$ si $i$ ou $j$ est \'egal \`a $(N+1)/2$ et de type $R_3$ si 
$i=N+1-j$.

\subsection{Application norme} \label{sec:Norme}
Posons
$$(1-\theta ) \T := \{ t \theta (t^{-1}) \; : \; t \in \T \}.$$ 
Kottwitz et Shelstad d\'efinissent une application norme sur $\T$ comme
l'application quotient
$$\mathrm{N} : \T \rightarrow \T / (1-\theta ) \T.$$
Posons
$$\T^{\perp} = \{ t \in \T \; : \; t \theta (t) \ldots \theta^{d-1} (t) = 1 \}.$$
Il est clair que 
$$\T^{\perp} = (1-\theta ) \T := \{ t \theta (t^{-1}) \; : \; t \in \T \}$$ 
et que l'application 
$$\T^1 \times \T^{\perp} \rightarrow \T , \ \ \ \Psi (t,h) = t h^{-1} \theta (h) $$
est surjective et a un noyau fini. 

\'Etant donn\'e une racine $\alpha \in R(\B , \T)$, notons $\mathrm{N}\alpha$ la somme des racines dans 
la $\theta$-orbite de $\alpha$~:
$$\mathrm{N} \alpha = \sum_{i=0}^{d_{\alpha} - 1} \theta^i (\alpha )$$
le nombre de racines dans l'orbite \'etant $d_{\alpha}$. 
Le caract\`ere $\mathrm{N}\alpha$ de $\T$ se factorise en un caract\`ere de $\T / (1-\theta ) \T$.
Et les applications 
$$\T^1 \rightarrow \T \rightarrow \T / (1 - \theta ) \T$$
permettent de r\'ealiser $\alpha_{\rm res}$ et $\mathrm{N}\alpha$ comme des caract\`eres d'un m\^eme tore $\T^1$; on a alors~:
$$\mathrm{N}\alpha = d_{\alpha} \alpha_{\rm res}.$$
L'ensemble des caract\`eres $\mathrm{N}\alpha$ forme un syst\`eme de racines {\it r\'eduit} $\mathrm{N}R$ pour $\T^1$. 

L'automorphisme $\theta$ induit un automorphisme $\widehat{\theta}$ sur le dual $\widehat{G}$, voir \cite[\S 1.2]{KS}. Fixons un \'epinglage $(\mathcal{B} , \mathcal{T} , \{ \mathcal{X} \})$ de $\widehat{G}$
pr\'eserv\'e par $\widehat{\theta}$. Alors l'ensemble des racines restreintes $R_{\rm res} (\widehat{G} , \mathcal{T})$ s'identifie \`a 
$$R^{\vee}_{\rm res} = \{ (\alpha^{\vee} )_{\rm res} = \alpha^{\vee} |_{\widehat{T}_{\widehat{\theta}}} \; : \; \alpha^{\vee} \in R^{\vee} (G , T) \}.$$
Et  $(\alpha^{\vee} )_{\rm res} \in X^* (\widehat{T}^{\widehat{\theta}}) = X^* (\widehat{T})_{\widehat{\theta}}$ a pour coracine un \'el\'ement de $X^* (T)^{\theta}$. Si $\alpha_{\rm res}$ est de type $R_1$
ou $R_3$ alors la coracine de $(\alpha^{\vee} )_{\rm res}$ est $\mathrm{N}\alpha$ sinon (type $R_2$) c'est $2\mathrm{N}\alpha$, voir \cite[(1.3.9)]{KS}. L'ensemble de ces caract\`eres forme encore un syst\`eme de racines r\'eduit pour $T^1$. Notons $E$ le groupe d\'eploy\'e correspondant.

Notons qu'au niveau des alg\`ebres de Lie, il correspond \`a $\mathrm{N}$ une application norme
\begin{equation} \label{N1}
\mathrm{N} : \mathfrak{t}_{\CC} \rightarrow \mathfrak{t}_{\CC}^1
\end{equation}
\'egale \`a la projection sur $\mathfrak{t}^1_{\CC}$ suivant la d\'ecomposition $\mathfrak{t}_{\CC} = 
\mathfrak{t}^1_{\CC} \oplus (1-\theta) \mathfrak{t}_{\CC}$. 

Consid\'erons maintenant un \'el\'ement semisimple $\delta = \theta g \in \LL$. Notons qu'il revient au 
m\^eme de consid\'erer la classe de conjugaison de $\delta$ par des \'el\'ements de $\G$ ou la classe de $\theta$-conjugaison de l'\'el\'ement $\theta$-semisimple $g$. Il d\'ecoule alors par exemple de 
\cite[Lem. 3.2.A]{KS} que la classe de $\theta$-conjugaison de $g$ est repr\'esent\'ee par un \'el\'ement 
$t \in \T$ uniquement d\'efini modulo $W$.  \`A l'aide de l'application norme 
Kottwitz et Shelstad construisent ainsi une correspondance naturelle des classes de conjugaison semisimples dans $E$ vers les classes de $\theta$-conjugaison $\theta$-semisimples dans $\G$. 
On la note $\mathcal{A}$ comme Kottwitz et Shelstad, voir \cite[Thm. 3.3.A]{KS}. On notera
$\mathcal{N}$ sa r\'eciproque. 

\medskip
\noindent
{\it Exemple.} Soit $\G$ le groupe des automorphismes lin\'eaires de $\CC^N$ et $\theta$ l'automorphisme involutif op\'erant par $g \mapsto g^{\theta}=J {}^t \- g^{-1} J^{-1}$. Le tore $\T \cong (\CC^*)^{N}$ est diagonal et si $t=(t_1 , \ldots , t_{N}) \in T$ on a~:
$$t\theta (t^{-1}) = (t_1t_{N} , \ldots , t_N t_1).$$
Par cons\'equent, 
$$(1-\theta) \T =  \left\{
\begin{split}
& \{ (u_1 , \ldots , u_{\ell} , u_{\ell} , \ldots , u_1 ) \; : \; u_i \in \CC^* \} \cong (\CC^*)^{\ell}  \mbox{ si } N=2\ell \\ 
& \{ (u_1 , \ldots , u_{\ell} , u_{\ell +1} , u_{\ell} , \ldots , u_1 ) \; : \; u_i \in \CC^* \} \cong (\CC^*)^{\ell+1} & \mbox{ si } N=2\ell +1 .
\end{split}
\right.$$
et $\T / (1-\theta) \T \cong (\CC^* )^{\ell}$. De sorte que l'application norme
$$\mathrm{N} : \T \rightarrow  (\CC^* )^{\ell} \cong \T^1$$
associe \`a $t= (t_1 , \ldots , t_N )$ l'\'el\'ement $(t_1 / t_N , \ldots , t_{\ell} / t_{N-\ell +1})$.

Les racines $\mathrm{N}\alpha$ du syst\`eme de racines $\mathrm{N}R$ sont celles du groupe $\SO (2\ell +1 , \CC)$. Le groupe $E$ est quant \`a lui \'egal \`a $\SO (2\ell+1 , \CC)$ si $N=2\ell$ est pair et \`a $\Sp (2\ell , \CC)$ si $N=2\ell+1$ est impair. On note $T_E$ le tore diagonal (pour les formes
classiques d\'eploy\'ees) de $E$; de sorte que les \'el\'ements de $T_E$ s'\'ecrivent~:
\begin{equation*}
\begin{split}
& \mathrm{diag} (x_1 , \ldots , x_{\ell} , x_{\ell}^{-1} , \ldots , x_{1}^{-1})  \ \ \ \mbox{ si } N \mbox{ est impair}, \\
& \mathrm{diag} (x_1 , \ldots , x_{\ell} , 1 , x_{\ell}^{-1} , \ldots , x_{1}^{-1} )  \ \ \ \mbox{ si } N \mbox{ est pair}.
\end{split}
\end{equation*}
{\it Via} ces  coordonn\'ees $T_E$ et $T^1$ sont canoniquement isomorphes. 

Remarquons que $E$ est un sous-groupe endoscopique maximal du groupe $\G \rtimes \langle \theta \rangle$ (cf. \cite{Arthur}); son
groupe dual est $\widehat{E}= \Sp (N , \CC)$, resp. $\SO (N, \CC)$, naturellement plong\'e dans 
$\widehat{G} = \GL_N (\CC)$.

L'application $\mathcal{A}$ est explicitement d\'ecrite par Waldspurger dans \cite[\S III.2]{Waldspurger}.\footnote{La diff\'erence de signe ici est due au fait que notre $\theta$ pr\'eserve un \'epinglage.} Les classes de $\theta$-conjugaison $\theta$-semisimples de $\G$ sont repr\'esent\'ees par les 
\'el\'ements 
\begin{equation}
t = \mathrm{diag} (s , 1 )\in T  \  (s \in (\CC^* )^{\ell})
\end{equation}
modulo $W=\mathfrak{S}_{\ell} \rtimes \{\pm 1 \}^{\ell}$.
Les classes de conjugaison semisimples de $E$ sont quant \`a elles repr\'esent\'ees par 
\begin{equation}
\begin{split}
& t ' = \mathrm{diag} (x_1 , \ldots , x_{\ell} , x_{\ell}^{-1} , \ldots , x_{1}^{-1}) \in T_E \ \ \ \mbox{ si } N \mbox{ est impair}, \\
& t' = \mathrm{diag} (x_1 , \ldots , x_{\ell} , 1 , x_{\ell}^{-1} , \ldots , x_{1}^{-1} ) \in T_E \ \ \ \mbox{ si } N \mbox{ est pair}
\end{split}
\end{equation}
modulo $W=\mathfrak{S}_{\ell} \rtimes \{\pm 1 \}^{\ell}$ -- le groupe de Weyl de $E$.
On a alors $t = \mathcal{A} (t')$ si $s= (x_1 , \ldots , x_{\ell})$. L'application $\mathcal{A}$ est bijective et 
$\mathcal{N}$ est l'application r\'eciproque. En particulier en restriction \`a $T^1$, l'application 
$\mathcal{N}$ s'\'ecrit~:
\begin{multline}
\begin{split}
\mathrm{diag} (x_1 , \ldots , x_{\ell} ,  & x_{\ell}^{-1} , \ldots , x_{1}^{-1})  \in T^1 \\  & \mapsto  \mathrm{diag} (x_1^2 , \ldots , x_{\ell}^2 , 1 , x_{\ell}^{-2} , \ldots , x_{1}^{-2} ) \in T_E \ \ \ \mbox{ si } N \mbox{ est pair}, \\
\mathrm{diag} (x_1 , \ldots , x_{\ell} , & 1 , x_{\ell}^{-1} , \ldots , x_{1}^{-1} ) \in T^1 \\ & \mapsto 
\mathrm{diag} (x_1^2 , \ldots , x_{\ell}^2 , x_{\ell}^{-2} , \ldots , x_{1}^{-2} ) \in T_E \ \ \ \mbox{ si } N \mbox{ est impair}.
\end{split}
\end{multline}

\subsection{D\'enominateur de Weyl tordu} \label{1.4} Soit $G$ (resp. $T$, {\it etc}...) l'ensemble des points r\'eels de $\G$ (resp. $\T$, {\it etc}...) et 
soit $\g$ (resp. $\mathfrak{t}$, {\it etc}...) son alg\`ebre de Lie (r\'eelle). On note $\widetilde{T}^1$ le rev\^etement universel de $T^1$.

L'ensemble des racines restreintes non nulles $\alpha_{\rm res}$, pour $\alpha$ racine de $\T$ dans 
$\G$, forme un syst\`eme de racine (non r\'eduit) et $R_{\rm res} $ est un choix de racines positives. Si $\alpha_{\rm res}$ est une racine dans $R_{\rm res}$ on note $\g^{\alpha_{\rm res}}_{\CC}$ l'espace radiciel correspondant. On pose 
$$\rho = \frac12 \sum_{\alpha_{\rm res} \in R_{\rm res} } ( \dim \g^{\alpha_{\rm res}}_{\CC} )\alpha_{\rm res} \ \ \ \mbox{ et } \ \ \ \uu^+ = \sum_{\alpha_{\rm res} \in R_{\rm res}} \g^{\alpha_{\rm res}}_{\CC}.$$
Alors $\uu^+$ est une sous-alg\`ebre de 
$\g_{\CC}$. Comme $\theta$ laisse stable $\uu^+$ et fixe $T^1$, 
il lui correspond un {\it d\'enominateur de Weyl tordu}
$$\Delta_{\theta} = \Delta_{G, \theta} : \widetilde{T}^1 \rightarrow \CC,$$ 
\'egal -- par d\'efinition -- au d\'eterminant 
$$D_{\uu^+}^{\theta} (t) = \det \left( (1 - {\rm Ad} (\theta t))_{|\uu^+} \right)$$
multipli\'e par $-\rho$~:
$$\Delta_{\theta} (t) = e^{-\rho} (t) D_{\uu^+}^{\theta} (t).$$  
Formellement $\Delta_{\theta}$ est \'egal au produit 
$\prod_{\alpha_{\rm res} \in R_{\rm res}} \prod_{j=1}^{n_{\alpha_{\rm res}}} (e^{-\alpha_{\rm res} /2} - \lambda_{\alpha_{\rm res} , j} e^{\alpha_{\rm res} /2})$,
o\`u $\lambda_{\alpha_{\rm res},1}, \ldots , \lambda_{\alpha_{\rm res} , n_{\alpha_{\rm res}}}$ sont
les valeurs propres de $\theta$ dans $\g^{\alpha_{\rm res}}_{\CC}$, compt\'ees avec multiplicit\'es.

Si $V$ est l'espace d'une repr\'esentation de $\widetilde{T}^1$ on note $\det [1 - V]$ la repr\'esentation
virtuelle (c'est-\`a-dire un \'el\'ement du $K_0$ dans la cat\'egorie des repr\'esentations de $\widetilde{T}^1$) d\'efinie par la somme altern\'ee $\sum_i (-1)^i [\wedge^i V]$ des puissances ext\'erieures. C'est
multiplicatif dans le sens naturel suivant~:
\begin{equation} \label{detmult}
\det [1-V \oplus W] = \det [1- V] \otimes \det [1-W] .
\end{equation}

Puisque $\g_{\CC} / \mathfrak{t}_{\CC} = \uu^+ \oplus \uu^-$, l'identit\'e \eqref{detmult} implique que le caract\`ere de 
$\det[1 -  \g_{\CC} / \mathfrak{t}_{\CC}]$ est \'egal \`a $\prod_{\alpha_{\rm res} \in R_{\rm res}} (1 - e^{\alpha_{\rm res}}) \cdot \prod_{\alpha_{\rm res} \in R_{\rm res}} (1 - e^{-\alpha_{\rm res}})$. 
Comme $\theta$ fixe $T^1$, on obtient facilement~:
\begin{equation} \label{WD}
\Delta_{\theta}^2 (t) =   \frac{(-1)^{\dim  \uu^+}}{\det (1-\theta)_{|\mathfrak{t}_{\CC} / \mathfrak{t}^1_{\CC}}} 
\det (1 - {\rm Ad}( \theta t))_{| \g_{\CC} / \mathfrak{t}^1_{\CC}} .
\end{equation}

\medskip
\noindent
{\it Exemple.}(Voir Chenevier-Clozel \cite{ChenevierClozel}.) Soit $\G$ le groupe des automorphismes lin\'eaires de $\CC^N$ et $\theta$ l'automorphisme involutif op\'erant par $g \mapsto g^{\theta}=J {}^t \- g^{-1} J^{-1}$. On prend
$\mathfrak{u}^+$ \'egale au radical unipotent de l'alg\`ebre de Borel standard de $\mathrm{M}_N (\CC)$
et 
$$t = \left( 
\begin{array}{ccccc}
t_1 & & & & \\
& t_2 & & & \\
& & \ddots & & \\
& & & t_2^{-1} & \\
& & & & t_1^{-1} 
\end{array} \right).$$
On calcule les valeurs propres sur $\mathfrak{u}^+$ de l'endomorphisme 
$$X \mapsto - J \mathrm{Ad} (t^{-1}) {}^t \- X J^{-1}.$$
Un calcul simple donne alors
\begin{multline} \label{DenEx}
D_{\mathfrak{u}^+}^{\theta} (t) = 
 \left\{
\begin{split}
& \prod_{i=1}^{\ell} (1 - t_i^{2}) \prod_{1 \leq i < j \leq \ell} 
(1 - t_i^2 /t_j^2) (1- t_i^{2} t_j^{2})  \mbox{ si } N= 2\ell \\
& \prod_{i=1}^{\ell} (1-t_i^{4})  \prod_{1 \leq i < j \leq \ell} 
(1 - t_i^2 /t_j^2) (1- t_i^{2} t_j^{2})  \mbox{ si } N=2\ell +1 
\end{split} \right.
\end{multline}
qui n'est autre que le d\'enominateur de Weyl $\prod (1-e^{\alpha})$ --  le produit 
portant sur les racines positives du groupe $E$ \'egal \`a $\SO (2 \ell +1 , \CC)$ si $N=2\ell$ est pair et \`a $\Sp (2\ell , \CC )$ si $N=2\ell +1$ est impair -- \'evalu\'e en $(t_i^{2})$.  Le terme $e^{-\rho} (t)$ le transforme en $D_E := \prod (e^{\alpha /2} - e^{-\alpha /2})$ \'evalu\'e en $\mathcal{N} (t)$.

\subsection{Repr\'esentations de dimension finie}
Soit $X^* (\T)$ le groupe des caract\`eres de $\T$. Il contient le r\'eseau des poids de $\G$
que nous notons $X^*_{\rm p}$. On note 
$X^*_{{\rm p} +}$ le sous-ensemble des \'el\'ements de $X^*_{\rm p}$ qui sont dominants relativement \`a l'ordre d\'efini par la paire $(\B , \T)$. Les \'el\'ements de $X^*_{{\rm p} +}$ param\`etrent 
les repr\'esentations irr\'eductibles de dimension finie de $\G$~: soit $\Lambda \in X^*_{{\rm p} +}$, alors 
$\Lambda$ est le plus haut poids (relativement \`a la paire $(\B , \T)$) d'une repr\'esentation 
de dimension finie $\tau_{\Lambda}$ de $\G$.

L'application norme identifie $X^* (\T^{1})$ \`a un sous-groupe de $X^* (\T)$. Avec cette identification
$X^* (\T^1)$ co\"{\i}ncide avec le sous-groupe des caract\`eres invariants par $\theta$.  La repr\'esentation $\tau_{\Lambda}$ est $\theta$-invariante ($\tau_{\Lambda} \cong \tau_{\Lambda}^{\theta}$ o\`u $\tau_{\Lambda}^{\theta} = \tau \circ \theta$) si et seulement si $\Lambda \in X^* (\T^1)$. 

Soit $X_{\rm p}^{\rm res} (\T^1) =  X^*_{{\rm p}} \cap X^* (\T^1)$ et soit $\Lambda \in X_{\rm p}^{\rm res} (\T^1) \cap X^*_{{\rm p} +} $. 
Le choix d'un op\'erateur $A_{\theta}$ ($A_{\theta}^d = 1$) dans l'espace $V_{\Lambda}$ de 
$\tau_{\Lambda}$ entrela\c{c}ant $\tau_{\Lambda}$ et $\tau_{\Lambda} \circ \theta$ permet d'\'etendre 
$\tau_{\Lambda}$ en une repr\'esentation irr\'eductible
$\tau^+$ de $\G \rtimes \langle \theta \rangle$; on obtient par ce proc\'ed\'e toutes les repr\'esentations irr\'eductibles de
dimension finie de $\G \rtimes \langle \theta \rangle$ dont la restriction \`a $\G$ est irr\'eductible,
voir \cite{KnappVogan} pour plus de d\'etails sur les repr\'esentations de dimension finie de groupes
non-connexes. 

Notons $\Theta_{\tau_{\Lambda} , \theta}$ le caract\`ere de $\tau^+$ sur $L$. Alors 
\begin{equation} \label{DF}
\Theta_{\tau_{\Lambda} , \theta} (t) = \sum_{\mu \in X_{\rm p}^{\rm res} (\T^1 )} m_{\mu} e^{\mu} (t) \ \ \ (t \in T^1),
\end{equation}
o\`u $m_{\mu}$ est la trace de $\theta$ sur l'espace propre de la valeur propre $\mu$. Noter que les caract\`eres qui ne sont pas $\theta$-invariants ne contribuent pas \`a la trace tordue. 


\section{Caract\`eres tordus de $G$}

Dans cette section nous expliquons bri\`evement, en suivant Bouaziz \cite{Bouaziz}, comment \'etendre le travail de Harish-Chandra \cite{HC1} \`a $L$ --  l'ensemble des points r\'eels de $\LL$. Soit $G^+$ le groupe obtenu en
formant le produit semi-direct de $\langle \theta \rangle \cong \Z / d \Z$ par $G$. On verra $L$ comme
la composante connexe -- \'egale \`a $\theta G$ -- de $\theta$ dans $G^+$.

On a
$$(1 , g) (\theta , x) (1 , g^{-1}) = (\theta , g x \theta (g^{-1})) = (\theta , g x (g^{\theta})^{-1}).$$
\'Etudier l'action de $G$ sur $L$ revient donc \`a \'etudier l'action tordue de $G$ sur lui-m\^eme. Par transport de structure on d\'efinit les \'el\'ements \og$\theta$-semisimples\fg \
et \og$\theta$-r\'eguliers\fg \ dans $G$.

On identifie l'alg\`ebre enveloppante
$U(\mathfrak{g}_{\CC})$ avec l'alg\`ebre des op\'erateurs diff\'erentiels invariants \`a gauche sur $L$. 
On note $Z(\mathfrak{g}_{\CC})$ le centre de $U(\mathfrak{g}_{\CC})$; on le consid\`ere donc comme
une alg\`ebre d'op\'erateurs diff\'erentiels $G$-bi-invariants sur $L$.

Soit $\Theta$ une distribution $G$-invariante sur $L$, distribution propre 
pour $Z(\g_{\CC})$. D'apr\`es Bouaziz \cite[Thm. 2.1.1]{Bouaziz} $\Theta$ est une fonction analytique sur 
$L_{\rm reg}$ -- l'ensemble des \'el\'ements r\'eguliers de $L$. 

\subsection{} \label{ChgtAutom}
Soit $\delta \in L$ semisimple. Quitte \`a conjuguer $\delta$ par un \'el\'ement
de $\G$, on peut supposer que $\delta$ fixe la paire $(\B , \T)$. Alors ${\rm Ad}_{\LL} (\delta)$
et $\theta$ diff\`erent par un automorphisme int\'erieur qui fixe la paire $(\B , \T)$. Autrement dit~:
$${\rm Ad}_{\LL} (\delta ) = {\rm Ad}_{\G} (t) \circ \theta$$ 
pour un certain $t \in T$ et donc $T^{\delta} = T^{\theta}$. 

Dans la suite nous supposerons donc que $T^{\delta} = T^{\theta}$. Nous supposerons de plus 
que $\delta$ est {\it r\'egulier}. 

Pour $t\in T^1$ suffisamment petit, l'\'el\'ement $\delta t$ est encore r\'egulier dans $L$. 
L'extension recherch\'ee du th\'eor\`eme d'Harish-Chandra est la d\'etermination de 
$t \mapsto \Theta (\delta t)$ sur un petit voisinage de l'identit\'e dans $T^1$.

\subsection{} Comme l'alg\`ebre $\mathfrak{t}_{\CC}$ est ab\'elienne, on peut identifier 
$U(\mathfrak{t}_{\CC})= Z(\mathfrak{t}_{\CC})$ avec l'alg\`ebre sym\'etrique $S(\mathfrak{t}_{\CC})$.
De m\^eme on identifie $Z(\mathfrak{t}^1_{\CC})$ avec $S(\mathfrak{t}^1_{\CC})$. 

L'application \eqref{N1} s'\'etend naturellement en une application norme 
\begin{equation} \label{N2}
\mathrm{N} : Z(\mathfrak{t}_{\CC}) \rightarrow Z(\mathfrak{t}^1_{\CC}).
\end{equation} 
On a par ailleurs des homomorphismes d'Harish-Chandra
$$Z(\mathfrak{g}_{\CC}) \rightarrow Z( \mathfrak{t}_{\CC}) \ \ \ \mbox{ et } \ \ \ 
Z(\mathfrak{g}^1_{\CC}) \rightarrow Z( \mathfrak{t}^1_{\CC})$$
que nous noterons tous les deux $\varphi$. Rappelons que $\varphi$ est un isomorphisme sur son
image $S(\mathfrak{t}_{\CC})^{W(\G , \T)}$ (resp. $S(\mathfrak{t}^1_{\CC})^W$) constitu\'ee des
polyn\^omes invariants par l'action du groupe de Weyl.

Par dualit\'e tout caract\`ere de $Z(\mathfrak{g}_{\CC})$ est de la forme 
\begin{equation} \label{dualite}
\chi_{\lambda} : Z(\mathfrak{g}_{\CC}) \rightarrow \C, \ \ \ \mbox{ avec } \chi_{\lambda} (Z) = \varphi (Z) (\lambda ),
\end{equation}
pour un certain $\lambda \in \mathfrak{t}_{\CC}^*$, 
et deux de ces caract\`eres co\"{\i}ncident pr\'ecis\'ement 
quand leur param\`etres appartiennent \`a la m\^eme $W(\G,\T)$-orbite. L'automorphisme 
$\theta$ op\`ere naturellement sur l'ensemble des caract\`eres de $Z(\mathfrak{g}_{\CC})$ et 
sur $\mathfrak{t}^*_{\CC}$. Ces actions sont compatibles avec \eqref{dualite}. 
Noter que l'application norme
$\mathrm{N}$ induit une injection de $\mathfrak{t}^{1  *}_{\CC}$ dans $\mathfrak{t}^*_{\CC}$; son image
est constitu\'e des \'el\'ements $\theta$-stables dans $\mathfrak{t}^*_{\CC}$.

Remarquons finalement que l'application \eqref{N2} s'\'etend de mani\`ere unique en 
une application norme $\mathrm{N} : Z( \mathfrak{g}_{\CC}) \rightarrow Z( \mathfrak{g}^1_{\CC})$ 
de sorte que le diagramme
$$\begin{CD}
Z(\mathfrak{t}^1_{\CC}) @> \varphi >>  Z(\mathfrak{g}^1_{\CC}) \\
@AA\mathrm{N}A  @AA\mathrm{N}A  \\
Z(\mathfrak{t}_{\CC}) @> \varphi >>  Z(\mathfrak{g}^1_{\CC}) 
\end{CD}
$$
soit commutatif. Un caract\`ere $\chi_{\lambda}$ est alors $\theta$-stable si et seulement s'il existe 
$\lambda^1 \in \mathfrak{t}^{1 *}_{\CC}$ tel que $\chi_{\lambda} = \chi_{\lambda^1} \circ N$; la $W$-orbite de $\lambda^1$ est alors uniquement d\'efinie.

\subsection{Partie radiale} Pour $t \in T^1$, posons~:
$$\nu_{\delta} (t) = \det \left( 1 - {\rm Ad} (\delta t) \right)_{|\g_{\CC} / \mathfrak{t}^1_{\CC}}.$$
C'est une fonction analytique sur $T^1$. On pose 
$$\Omega^1 = \{ t \in T^1 \; : \; \nu_{\delta} (t) \neq 0\}.$$
C'est un ouvert dense dans $T^1$. Alors la sous-vari\'et\'e $\delta \Omega^1$ est localement
ferm\'ee dans $L$, constitu\'ee d'\'el\'ements r\'eguliers et l'application 
$$\psi : L \times \delta \Omega^1 \rightarrow L ; \ (x,y) \mapsto xyx^{-1}$$
est submersive. 
On dispose en outre 
d'une mesure naturelle sur les fibres (lisses) de l'application 
$\psi$. L'int\'egration le long des fibres nous donne une application de $C_c^{\infty} (L \times \delta \Omega^1 )$ dans $C_c^{\infty}$. Notons $\psi^*$ l'application duale sur les distributions.
Harish-Chandra \cite{HC1} montre que $\psi^* (\Theta) =1 \otimes \Theta_{| \delta \Omega^1}$ -- la
restriction \'etant prise comme fonction analytique. 

En fait, si $D$ est un op\'erateur diff\'erentiel bi-$G$-invariant dans $L$, 
Harish-Chandra \cite{HC1} montre 
de plus qu'il existe un op\'erateur $\Delta (D)$ sur $\delta \Omega^1$ satisfaisant 
\begin{equation} \label{Delta}
(Df)_{|  \delta \Omega^1} = \Delta (D) (f_{|  \delta \Omega^1} ),
\end{equation}
pour toute fonction $f$ $C^{\infty}$, $G$-invariante dans un voisinage de $ \delta \Omega^1$.
L'op\'erateur $\Delta (D)$ est appel\'e {\it partie radiale} de l'op\'erateur $D$.

Le th\'eor\`eme suivant est d\'emontr\'e par Bouaziz \cite[Thm. 2.4.1]{Bouaziz}. 
On choisit un voisinage ouvert connexe $\mathcal{V}$ de
$0$ dans $\mathfrak{t}^1$ tel que $\mathcal{W} = \exp (\mathcal{V})$ soit ouvert et inclus dans $\Omega^1$. On suppose de plus que pour tout $t \in \mathcal{W}$, l'\'el\'ement $\delta t$ est r\'egulier et que l'application $\exp H \mapsto e^{\rho (H)}$ est bien d\'efinie dans $\mathcal{W}$. 

\begin{thm} \label{radiale}
Soit $z \in Z(\g_{\CC})$. L'op\'erateur $\Delta (z)$ est uniquement d\'efini sur $\delta \mathcal{W}$, et~:
$$\Delta (z) = |\nu_{\delta} |^{-1/2} (\varphi \circ N ) (z) \circ |\nu_{\delta} |^{1/2}.$$
\end{thm}

Supposons maintenant que $Z(\mathfrak{g})$ agisse sur $\Theta$ par $\chi_{\lambda}$
pour un certain $\lambda \in \mathfrak{t}^*$. Le caract\`ere $\chi_{\lambda}$ est n\'ecessairement 
$\theta$-stable \'egal \`a $\chi_{\lambda^1} \circ N$. Notons 
$$F ( t ) = |\nu_{\delta} (t) |^{1/2} \Theta (\delta t) \ \ \ (t \in \mathcal{W}).$$
Il d\'ecoule du th\'eor\`eme \ref{radiale} -- voir \cite[Cor. 2.4.11]{Bouaziz} -- que pour tout 
$z \in Z(\mathfrak{g}_{\CC})$, $(\varphi \circ N ) (z) F = \chi_{\lambda} (z) F$.
La fonction $F$ v\'erifie donc les \'equations diff\'erentielles~:
\begin{equation} \label{ED}
z F = \lambda^1 (z) F \ \ \ \left( z \in S(\mathfrak{t}_{\CC}^1 )^W\right).
\end{equation}
Ces \'equations sont \'etudi\'ees et r\'esolues par Harish-Chandra dans \cite[pp. 130--133]{HC1}. Une solution de \eqref{ED} sur $\mathcal{W}$ est un polyn\^ome exponentiel~:
$$F(\exp H) = \sum_{w \in W} \sum_{w \in W} P_w (H) e^{\lambda^1 (w H)},$$
o\`u $H \in \mathcal{V}$ et les coefficients $P_{w} (H)$ sont des polyn\^omes en $H$ de degr\'e strictement inf\'erieur \`a l'ordre de $W_{\lambda^1}$ -- le stabilisateur de $\lambda^1$ dans $W$.

\subsection{Repr\'esentations $\theta$-stables} 
Consid\'erons maintenant une repr\'esentation admissible $(\pi, \mathcal{H}_{\pi})$ de $G$ qui poss\`ede un caract\`ere infinit\'esimal $\chi_{\lambda}$ ($\lambda \in \mathfrak{t}_{\CC}^*$). 
Supposons $\pi$ $\theta$-invariante ($\pi \cong \pi^{\theta}$ o\`u $\pi^{\theta} = \pi \circ \theta$) et 
notons $V$ son module d'Harish-Chandra. 
Le choix d'un op\'erateur d'entrelacement $A_{\theta} : \mathcal{H}_{\pi} \rightarrow \mathcal{H}_{\pi}$ ($A_{\theta}^d = 1$) entrela\c{c}ant $\pi$ et $\pi \circ \theta$ permet d'\'etendre $\pi$ en une 
repr\'esentation $\pi^+$ de $G^+$. Le caract\`ere de $\pi^+$ sur
$L = \theta G$ est un distribution $G$-invariante, distribution propre pour $Z(\mathfrak{g}_{\CC})$.
Les r\'esultats qui pr\'ec\`edent s'appliquent donc. Notons
$$\Theta_{\pi , \theta} (g) = \Theta_{\pi^+} (\theta g) \ \ \ \ \ \ (g \in G);$$
c'est {\it le caract\`ere tordu} de $\pi$.

\begin{thm} \label{T26}
Soit $x\in T^{\theta}$ un \'el\'ement $\theta$-r\'egulier dans $G$. Alors, il existe un voisinage $\mathcal{V}$ de $0$ dans $\mathfrak{t}^1$ tel que pour tout $H \in \mathcal{V}$ on a~:
$$\left[ |\Delta_{\theta}| \Theta_{\theta , \pi} \right] ( x \exp H) = \sum_{w \in W}  c_{\lambda} (x, w ) e^{\lambda^1 (w H)},$$
o\`u les $c_{\lambda} (x , w)$ sont des constantes dans $\CC$.
\end{thm}
\begin{proof} Posons $\delta = \theta x \in L$. C'est un \'el\'ement semisimple r\'egulier. On conserve les notations introduites ci-dessus de sorte que $\mathcal{V}$ est d\'ej\`a d\'efini et que pour $H \in 
\mathcal{V}$, on a~: 
\begin{equation*}
\begin{split}
F(\exp H) & :=  |\nu_{\delta} ( \exp H) |^{1/2} \Theta_{\pi^+} (\delta \exp H) \\
& = \sum_{w \in W} P_w (H) e^{\lambda^1 (w H)},\\
\end{split}
\end{equation*}
o\`u $H \in \mathcal{V}$ et les coefficients $P_{w} (H)$ sont des polyn\^omes en $H$ de degr\'e strictement inf\'erieur \`a l'ordre de $W_{\lambda^1}$ -- le stabilisateur de $\lambda$ dans $W$.

L'identit\'e \eqref{WD} implique qu'il existe une constante $c$ telle que 
$$|\Delta_{\theta} (x \exp H)| = c |\nu_{\delta} ( \exp H) |^{1/2} .$$
Pour conclure la d\'emonstration du th\'eor\`eme il reste donc \`a faire voir que 
les coefficients $P_w (H)$ sont des constantes. Pour ce faire nous reprenons une id\'ee 
de Fomin et Shapovalov \cite{FS}. 

Soit $\tau = \tau_{\Lambda}$ ($\Lambda \in X_{\rm p}^{\rm res} (\T^1) \cap X_{{\rm p}+}^*$) 
une repr\'esentation de dimension finie de $\G$ invariante par $\theta$. Alors
la repr\'esentation $\pi \otimes \tau$ est encore admissible et $\theta$-invariante. Le choix
d'un op\'erateur entrela\c{c}ant $\tau$ et $\tau \circ \theta$ permet en outre de parler des
caract\`eres tordus $\Theta_{\pi \otimes \tau , \theta}$ et $\Theta_{\tau , \theta}$. Et le caract\`ere tordu 
$\Theta_{\pi \otimes \tau , \theta}$ est \'egal au produit des caract\`eres tordus $\Theta_{\pi , \theta}$ et $\Theta_{\tau , \theta}$ de $\pi$ et $\tau$~:
\begin{equation} \label{1}
\Theta_{\pi \otimes \tau , \theta} = \Theta_{\tau , \theta} \Theta_{\pi , \theta}.
\end{equation}

D'un autre c\^ot\'e $\pi \otimes \tau$ est de longueur finie (voir \cite[Prop. 10.41]{Knapp}); notons 
$\pi_1 , \ldots , \pi_n$ la suite de Jordan-H\"older de $\pi \otimes \tau$.
Puisque $\pi \otimes \tau$ est $\theta$-invariante, l'unicit\'e des facteurs de composition implique 
que quitte \`a r\'eordonner les $\pi_j$ on peut supposer que les
repr\'esentations $\pi_1 , \ldots , \pi_r$, avec $r\leq n$, sont $\theta$-invariantes (\'el\'ements diagonaux) 
et que les autres ne contribuent pas au caract\`ere tordu. De sorte que 
le caract\`ere tordu de $\pi \otimes \tau$ se d\'ecompose
en une combinaison lin\'eaire des caract\`eres tordus des sous-quotients $\theta$-stables $\pi_1 , \ldots , \pi_r$~:
\begin{equation} \label{2}
\Theta_{\pi \otimes \tau , \theta} = \Theta_{\pi_1 , \theta} + \ldots + \Theta_{\pi_r , \theta}.
\end{equation}
Notons que nous n'avons pas regroup\'e les repr\'esentations \'equivalentes. Si on regroupe par $\pi_i$ \'equivalentes,
les \og multiplicit\'es \fg \ sont dans $\ZZ [\zeta_d ]$ o\`u $d$ d\'esigne toujours l'ordre de $\theta$. 

Il d\'ecoule de \eqref{1} et de \eqref{DF} que pour tout $H \in \mathcal{V}$, le caract\`ere tordu de
$\pi \otimes \tau$ est donn\'e par
\begin{equation*} 
\begin{split}
\Theta_{\pi , \theta} (x \exp H ) \Theta_{\tau , \theta} (x\exp H) & = \sum_{w \in W} P_w (H) e^{\lambda^1 (w H)} \cdot \sum_{\mu \in X_{\rm p}^{\rm res} (\T^1 )} e^{\mu} (x) m_{\mu} e^{\mu (H)} \\
& = \sum_{w \in W} \sum_{\mu \in X_{\rm p}^{\rm res} (\T^1 )} m_{\mu} e^{\mu} (x) P_w (H) e^{(w^{-1} \lambda^1 + \mu) (H)}  \\
& = \sum_{w \in W} \ \sum_{\mu \in X_{\rm p}^{\rm res} (\T^1 )} m_{w^{-1} \mu} e^{w^{-1} \mu} (x) P_w (H) e^{(\lambda^1 + \mu) (wH)}. 
\end{split}
\end{equation*}
D'un autre c\^ot\'e, on peut \'ecrire~:
$$\sum_{i=1}^r  \Theta_{\pi_i , \theta} (x \exp H) =  \sum_{i=1}^r \sum_{w \in W} P_{w , i} (H) e^{\lambda_i^1 (wH)} .$$

Il d\'ecoule en particulier \eqref{2} que pour tout $H \in \mathcal{V}$, on a~:
\begin{equation} \label{3}
 \sum_{w \in W} \ \sum_{\mu \in X_{\rm p}^{\rm res} (\T^1 )} m_{w^{-1} \mu} e^{w^{-1} \mu} (x) P_w (H) e^{(\lambda^1 + \mu) (wH)} = \sum_{i=1}^r \sum_{w \in W} P_{w , i} (H) e^{\lambda_i^1 (wH)} .
\end{equation}

On peut supposer $\mathrm{Re} (\lambda^1 )$ et tous les $\mathrm{Re} (\lambda_i^1)$ ($i=1 , \ldots , r$) dominants. On 
peut en outre choisir $\Lambda$ de telle sorte que le stabilisateur $W_{\Lambda} \subset W$ soit
trivial. Notons que dans ce cas $\lambda^1 + \Lambda$ est encore dominant et que le stabilisateur
$W_{\lambda^1 + \Lambda}$ est encore trivial. Il d\'ecoule alors de \eqref{3} que 
\begin{equation}
m_{w^{-1} \mu} e^{w^{-1} \Lambda} (x)  P_w (H) = \sum_i P_{w , i } (H).
\end{equation}
Ici la somme porte sur les indices $i$ pour lesquels $\lambda_i^1 =  \lambda^1 + \Lambda$. Mais
le stabilisateur $W_{\lambda^1 + \Lambda}$ est encore trivial. Les polyn\^omes $P_{w , i}$ 
sont donc des constantes et le th\'eor\`eme est d\'emontr\'e.
\end{proof}

\subsection{} 
Soit $K$ un sous-groupe compact maximal de $G$ et une d\'ecomposition de Cartan 
\begin{equation} \label{decCartan}
\mathfrak{g} = \mathfrak{k} \oplus \mathfrak{p}
\end{equation} 
de sorte que l'involution de Cartan correspondante $\theta_{\rm Cartan}$ commute \`a $\theta$ (un tel
$K$ existe car $\theta$, d'ordre fini, fixe un point de l'espace sym\'etrique de $G$). Alors
la d\'ecomposition de Cartan \eqref{decCartan} est $\theta$-invariante.  
Notons $\mathfrak{t}_c^1 = \mathfrak{t}^1 \cap \mathfrak{k}$ et $\mathfrak{a}^1 = \mathfrak{t}^1
\cap \mathfrak{p}$. Soit $T_c^1$ et $A^1$ les sous-groupes analytiques de $G$ correspondants; on a~\footnote{Rappelons que $T^1$ est connexe.}
$T^1=T_c^1 A^1$ avec $T_c \subset K$. Une racine restreinte $\alpha_{\rm res}$ est dite
r\'eelle (resp. imaginaire) si $\alpha_{\rm res}$ est nulle sur $\mathfrak{t}_c^1$ (resp. $\mathfrak{a}^1$).
On note 
$$\mathfrak{a}^{1 -} = \{ H \in \mathfrak{a}^1 \; : \; \alpha_{\rm res} (H) < 0, \ \forall \alpha_{\rm res} \in R_{\rm res} \mbox{ non-imaginaire} \}.$$
Noter que l'ensemble des \'el\'ements $\theta$-r\'eguliers dans $T_c^1 \exp \mathfrak{a}^{1 -}$ n'est pas connexe.
Dans ce paragraphe on cherche \`a prolonger l'identit\'e du th\'eor\`eme \ref{T26} \`a tout cet ensemble.

D'apr\`es \cite[Lem. 3.6.3]{Bouaziz}, il existe une fonction localement constante 
$\varepsilon$ sur l'ensemble des \'el\'ements $\theta$-r\'eguliers dans $T^1$ telle que 
\begin{equation} \label{3627}
|\Delta_{\theta} (\exp H) | = \varepsilon (\exp H) e^{-\rho (H)} D_{\uu^+}^{\theta} (\exp H )
\end{equation} 
pour tout $H \in \mathfrak{t}^1$ tel que $\exp H$ soit $\theta$-r\'egulier.

Soit $\mathcal{V}$ un ouvert connexe dans $\mathfrak{t}_c^1 + \mathfrak{a}^{1,-}$. On fixe $H_0 \in 
\mathcal{V}$. D'apr\`es le th\'eor\`eme \ref{T26} et \eqref{3627} on a alors~:
$$\varepsilon (\exp H_0) e^{-\rho (H)} \left[ D_{\n^+}^{\theta} \Theta_{\theta , \pi} \right] (\exp H) = \sum_{
\mu \in W \cdot \lambda^1} C_{\mu} e^{\mu (H)} \ \ \ (H \in \mathcal{V})$$
o\`u les $C_{\mu}$ sont des constantes complexes (qui d\'ependent du choix de $H_0$).

D'apr\`es \cite[Lem. 3.6.23]{Bouaziz}, la fonction 
$H \mapsto \left[ D_{\uu^+}^{\theta} \Theta_{\theta , \pi} \right] (\exp H)$ se prolonge analytiquement \`a 
$\mathfrak{t}_c^1 + \mathfrak{a}_1^-$. On a donc~:
$$\left[ D_{\uu^+}^{\theta} \Theta_{\theta , \pi} \right] (\exp H) = \varepsilon (\exp H_0)^{-1} 
e^{\rho (H)}\sum_{\mu \in W \cdot \lambda^1} C_{\mu} e^{\mu (H)} \ \ \ (H \in \mathcal{V})$$
pour tout $H \in \mathfrak{t}_c^1 + \mathfrak{a}_1^-$. En r\'eutilisant \eqref{3627} on obtient~:
\begin{equation} \label{3628}
\left[ |\Delta_{\theta} | \Theta_{\theta , \pi} \right] (\exp H)  = \frac{\varepsilon (\exp H)}{\varepsilon (\exp H_0 )} \sum_{\mu 
\in W \cdot \lambda^1} C_{\mu} e^{\mu (H)}
\end{equation}
pour tout $H \in \mathfrak{t}_c^1 + \mathfrak{a}_1^-$ tel que $\exp H$ soit $\theta$-r\'egulier.

\subsection{Caract\`eres des repr\'esentations induites} \label{par:38}

Dans ce paragraphe nous supposons le tore $T$ maximalement d\'eploy\'e sur $\RR$ et notons $A$
sa partie d\'eploy\'ee. Nous supposons en outre que $T$ est contenu dans un sous-groupe parabolique
minimal $P=MAN$ de $G$ qui est invariant par $\theta$.\footnote{Noter qu'un tel $P$ n'existe pas en g\'en\'eral~: penser au cas o\`u $G$ est un groupe
complexe, vu comme groupe r\'eel, et o\`u $\theta$ correspond \`a la conjugaison complexe par rapport
\`a une forme r\'eelle non quasi-d\'eploy\'ee. (Ce probl\`eme apparent dispara\^{\i}t dans le cadre des espaces 
tordus.)}
Soit $(V_M , \tau)$ une repr\'esentation irr\'eductible $\theta$-stable de $M$ et $\mu$ un 
caract\`ere $\theta$-stable de $A$. On pose 
$$I = {\rm ind}_{MAN}^G (V_M \otimes \CC_{\mu} )$$
(induite unitaire).
La repr\'esentation $I$ est de longueur finie mais 
peut \^etre r\'eductible. N\'eanmoins le choix d'un op\'erateur $a_{\theta}$ ($a_{\theta}^d=1$) dans l'espace $V_M$ entrela\c{c}ant $\tau$ et $\tau \circ \theta$ est uniquement d\'efini \`a une racine de l'unit\'e pr\`es et induit un op\'erateur $A_{\theta} : I \rightarrow I$ par la formule
$$A_{\theta} (f) (g) = a_{\theta} f (\theta^{-1} (g)) \quad (\mbox{o\`u } f(gm) = \tau (m^{-1} f(g)).$$
L'op\'erateur $A_{\theta}$ d\'efinit une action de $\theta$ sur
$I$; dans la suite on choisira toujours cette action (uniquement d\'efinie \`a une racine de l'unit\'e pr\`es); notons en particulier que $A_{\theta}^d=1$. 
On peut alors consid\'erer le caract\`ere
tordu $\Theta_{I , \theta}$; c'est une fonction localement sommable analytique
sur les \'el\'ements $\theta$-r\'eguliers de $G$. Le lemme
suivant d\'ecoule de \cite[Lem. 7.1.3]{Bouaziz} qui g\'en\'eralise un th\'eor\`eme d'Hirai \cite[Thm. 2]{Hirai} dans le cas non tordu.

Le groupe $G$ op\`ere par $\theta$-conjugaison sur lui-m\^eme; on note 
\begin{equation*}
\begin{split}
N_{G, \theta} (T) & = \{ g \in G \; : \; gT^1(g^{\theta})^{-1} \subset T^1 \} \\
& = \{ g \in N_G (\mathfrak{t}^1) \; : \; g (g^{\theta} )^{-1} \in T^1 \} 
\end{split}
\end{equation*}
et $Z_{G ,\theta} (T) = T^{\theta}$.
Le groupe de Weyl tordu  
$$W_{\theta} (G , T) = N_{G , \theta} (T) / T^{\theta}$$
op\`ere sur $\mathfrak{t}$ et s'identifie naturellement \`a un sous-groupe d'indice fini de $W$.
On dispose de la m\^eme mani\`ere d'un groupe de Weyl tordu $W_{\theta} (MA , T)$. 

Puisque $P$ est minimal, si $T$ et $T'$ sont deux tores maximalement d\'eploy\'es dans $MA$ alors
$T^1$ et $(T')^1$ sont $\theta$-conjugu\'es dans $MA$.
L'expression du caract\`ere tordu de $I$ est donc particuli\`erement simple dans notre cas~:

\begin{lem} \label{Hirai}
Pour tout \'el\'ement $\theta$-r\'egulier $t\in T^1$, on a~:
$$\left[ |\Delta_{\theta} | \Theta_{I ,\theta}\right] (t) = \frac{1}{\# W_{\theta} (MA , T)} \sum_{v \in W_{\theta} (G , T)} \left[ |\Delta_{MA , \theta} | \Theta_{MA , \theta} (V_M) \right] (v t) .$$
\end{lem}  

Ici $\Theta_{MA , \theta} (V_M)$ d\'esigne la caract\`ere tordu de la repr\'esentation $V_M$ du groupe
r\'eductif $MA$. Noter que $V_M$ est une repr\'esentation $\theta$-stable et 
de dimension finie. Supposons que $Z(\mathfrak{m}_{\CC} \oplus \mathfrak{a}_{\CC} )$ agisse sur 
$V_M$ selon $\chi_{\lambda}$ avec $\lambda \in \mathfrak{t}_{\CC}^*$. 
D'apr\`es le th\'eor\`eme \ref{T26}, au voisinage d'un point $x \in T^{\theta}$ 
$\theta$-r\'egulier dans $MA$, la fonction $ |\Delta_{MA , \theta} | \Theta_{MA , \theta} (V_M)$ s'exprime
sous la forme~:
\begin{equation} \label{59}
\sum_{w \in W(\mathfrak{m}_{\CC} \oplus \mathfrak{a}_{\CC} , \mathfrak{t}_{\CC})^{\theta}} d_{\lambda} (x , w) e^{w^{-1} \lambda^1},
\end{equation}
o\`u les $d_{\lambda} (x , w)$ sont des constantes dans $\CC$. 
Dans un voisinage de $x$ dans $T^{\theta}$ on a donc~:
\begin{multline} \label{514}
\left[ |\Delta_{\theta} | \Theta_{I ,\theta}\right] \\ = \frac{1}{\# W_{\theta} (MA , T)} \sum_{v \in W_{\theta} (G , T)} \sum_{w \in W(\mathfrak{m}_{\CC} \oplus \mathfrak{a}_{\CC} , \mathfrak{t}_{\CC})^{\theta}} d_{\lambda} (vx , w) e^{v^{-1} w^{-1} \lambda^1}.
\end{multline}

\section{Exposants tordus, $\n$-homologie et conjecture d'Osborne}
\label{S4}

Dans ce chapitre nous supposons, comme dans le \S \ref{par:38}, que 
le tore $T$ maximalement d\'eploy\'e sur $\RR$ est contenu dans un parabolique minimal $\theta$-stable $P$. Nous notons $A$ 
sa partie d\'eploy\'ee et $\mathfrak{n}$ la complexifi\'ee de l'alg\`ebre de Lie du radical unipotent 
de $P=MAN$. De sorte que $T=T_c A$, o\`u $A = (\RR_+^*)^r$ et $T_c$ est compact; l'entier $r$
est le rang r\'eel de $G$. L'automorphisme $\theta$
pr\'eserve $A$. Un caract\`ere $\chi$ de $A$ est dit {\it $\theta$-stable} si $\chi \circ \theta = \chi$.

\subsection{Exposants $\theta$-stables} \label{section:exp}
Notons $\a$ l'alg\`ebre de Lie de $A$. Un {\it exposant} est un caract\`ere de $A$.
Par abus de notation on identifie un caract\`ere de $A$ \`a un \'el\'ement de ${\rm Hom} (\a , \CC)=\CC^r$. L'alg\`ebre $\mathfrak{n}$ d\'etermine un syst\`eme de racines positives qui sont les parties r\'eelles 
des racines non-imaginaires dans $R=R(\B , \T)$. Pour deux 
exposants $e$, $e'$, on \'ecrit $e\leq e'$ si  
$$e' = e + \sum_{\alpha} n_{\alpha} \alpha \ \ \ (n_{\alpha} \geq 0)$$
o\`u $\alpha$ d\'ecrit les racines simples dans $\mathfrak{n}$ et $e$, $e'$ sont vus comme des formes lin\'eaires complexes sur $\a = \RR^r$. Noter qu'un exposant 
$$e = \sum_{\alpha} n_{\alpha} \alpha $$
est $\theta$-stable si et seulement si
\begin{eqnarray*}
e & = & \sum_{\alpha} \frac{1}{d} n_{\alpha} (\alpha +\theta(\alpha) + \ldots + \theta^{d-1} (\alpha)) \\
& = & \sum_{\alpha} \frac{1}{d_{\alpha}} n_{\alpha} \mathrm{N} \alpha \\
& = & \sum_{\mathrm{N} \alpha} n_{\alpha} \mathrm{N} \alpha .
\end{eqnarray*}
Et l'ordre naturel sur les racines dans $\mathrm{N}R$ (\S \ref{sec:Norme}) d\'efinit un ordre 
sur les exposants $\theta$-stables; dans la deuxi\`eme partie de cet article nous notons $\leq_{\theta}$ cet ordre. 

\medskip
\noindent
{\it Exemple.} Soit $\G$ le groupe des automorphismes lin\'eaires de $\CC^N$ et $\theta$ l'automorphisme involutif 
op\'erant par $g \mapsto g^{\theta}=J {}^t \- g^{-1} J^{-1}$. Le groupe $E$ \'egal \`a $\SO (2\ell+1, \CC)$, 
resp. $\Sp (2\ell , \CC)$, si $N=2\ell$ est pair, resp. si $N=2\ell+1$ est impair. 

{\it Via} l'application norme, un exposant $\theta$-stable $e \in \CC^N$ d\'efinit un exposant de $E$
c'est-\`a-dire un vecteur dans $\CC^{\ell}$. L'ordre naturel sur les racines de $E$ definit donc un ordre $\leq_{E}$ sur les exposants $\theta$-stables. Si $e \leq_{E} e'$ on a toujours $e \leq_{\theta} e'$ mais 
la r\'eciproque n'est vraie que si $N$ est pair. 
Pour deux exposants $\theta$-stables $e$, $e'$ l'ordre $\leq_{E}$ s'exprime par 
$e'-e=(x_i)_{i=1, \ldots , \ell}$ avec 
\begin{eqnarray} \label{ordre3}
\left\{
\begin{array}{l}
x_1 + \ldots + x_i \in \N \ \ \ (1 \leq i \leq \ell -1)  \\
\mbox{et} \\
x_1 + \ldots + x_{\ell} \in \N, \ \ \ {\rm resp.} \  2 \N 
\end{array}
\right.
\end{eqnarray}
si $N=2\ell$ est pair, resp. si $N=2\ell+1$ est impair. L'ordre $\leq_{\theta}$ s'exprime quant \`a lui par 
$e'-e=(x_i)_{i=1, \ldots , \ell}$ avec 
\begin{eqnarray} \label{ordre}
\left\{
\begin{array}{l}
x_1 + \ldots + x_i \in \N \ \ \ (1 \leq i \leq \ell -1)  \\
\mbox{et} \\
x_1 + \ldots + x_{\ell} \in \N
\end{array}
\right.
\end{eqnarray}
ind\'ependamment de la parit\'e de $N$. 

\subsection{Exposants d'homologie}

Soit $R_{\rm res}^{\rm n.i.} \subset R_{\rm res}$ le sous-ensemble constitu\'e des racines (restreintes) non-imaginaires. 

\begin{lem} \label{L:res}
Soit $\alpha \in R$ une racine non imaginaire alors $\alpha_{\rm res} \in R_{\rm res}^{\rm n.i.}$.
\end{lem}
\begin{proof} Puisque $\alpha$ est non imaginaire et que $P$ est minimal on a $\mathfrak{g}_{\CC}^{\alpha} 
\subset \mathfrak{n}$. Il s'agit donc de montrer qu'une racine (r\'eelle) $\alpha$ de $\mathfrak{a}$ dans $\mathfrak{n}$ se restreint non trivialement \`a $\mathfrak{a}^1$.
Les racines de $\mathfrak{a}$ dans $\mathfrak{n}$ d\'eterminent une chambre de Weyl aigu\"e (ouverte) positive $C$. Soit $H \in C$; alors pour toute racine  $\alpha$ 
de $\mathfrak{a}$ dans $\mathfrak{n}$ on a $\alpha (H) >0$. Mais puisque par ailleurs $\mathfrak{n}$ est $\theta$-stable $\theta$ permute les
racines (positives) de $\mathfrak{a}$ dans $\mathfrak{n}$. On a donc~:
$$\alpha \left(\sum_{i=0}^d \theta^i (H)\right) = \sum_{i=0}^d \theta^i (\alpha) (H) >0$$
et $\sum_{i=0}^d \theta^i (H) \in \mathfrak{a}^1$.
\end{proof}

Il d\'ecoule en particulier du lemme \ref{L:res} que  
$$\n = \sum_{\alpha_{\rm res} \in R_{\rm res}^{\rm n.i.}} \g_{\CC}^{\alpha_{\rm res}} \subset \uu^+ .$$ 
Notons 
$$\rho_P = \frac12 \sum_{\alpha_{\rm res} \in R_{\rm res}^{\rm n.i.}}  ( \dim \g^{\alpha_{\rm res}}_{\CC} )\alpha_{\rm res}.$$
 Soit 
$V$ le module d'Harish-Chandra d'une repr\'esentation (admissible) irr\'eductible $\theta$-stable $\pi$.
Un lemme de  Casselman et Osborne \cite{CasselmanOsborne} affirme que $V$ est de type fini comme 
$U(\n)$-module. 
Pour tout $q \geq 0$, le groupe d'homologie $H_q (\n , V)$ est donc de dimension finie. Si $e$ est 
un exposant, on peut consid\'erer le $(e+\rho_P)$-espace propre g\'en\'eralis\'e; notons-le
$H_q (\n , V)_{e}$. On dit que $e$ est un {\it exposant d'homologie} si 
$$H_q (\n , V)_e \neq 0$$
pour un certain $q$. La raison d'\^etre du d\'ecalage par $\rho_P$ deviendra claire au chapitre suivant, voir
\cite[p. 51]{HechtSchmid}.

Soit $q$ un entier positif. Alors $T$ et $\theta$ ({\it via} $A_{\theta}$) op\`erent sur $H_q (\n , V)$. 
Pour $t \in T$ (en particulier $t \in T^1$) on peut consid\'erer 
\begin{eqnarray} \label{HomTordue}
\Theta_q^{\theta} (t,V) = {\rm trace} \left( t \theta \ | \  H_q (\n , V) \right) .
\end{eqnarray}
Notons une diff\'erence importante avec le cas non tordu. L'op\'erateur $A_{\theta}$ op\`ere sur l'homologie
avec des valeurs propres (qui sont des racines de l'unit\'e). L'expression \eqref{HomTordue} est
donc (m\^eme pour $q$ fix\'e) une somme \`a coefficients dans les racines de l'unit\'e de caract\`eres 
$\theta$-stables de $t \in T$. Cette somme peut s'annuler sur $A$ sans \^etre nulle sur $T$.

\'Etant donn\'e un exposant $e$ on notera $f \in \a^{1 *}\otimes \CC$ le caract\`ere obtenu par restriction. De la m\^eme mani\`ere que les racines dans $R$ d\'efinissent un ordre sur les  exposants, les racines restreintes 
dans $R_{\rm res}$ d\'efinissent un ordre naturel sur les caract\`eres $f \in \a^{1 *}\otimes \CC$; nous notons $\leq_{\rm res}$ cet ordre.

Un caract\`ere $f \in \a^{1 *}\otimes \CC$ est un 
{\it exposant d'homologie restreint} s'il appara\^{\i}t dans $H_q (\n , V)$ pour quelque $q$. C'est un 
{\it exposant restreint minimal} si c'est un exposant d'homologie restreint et s'il est minimal pour $\leq_{\rm res}$. 
Rappelons un r\'esultat fondamental de Hecht-Schmid dans le cas non tordu~: \cite[Prop. 2.32]{HechtSchmid}. 

\begin{prop} \label{ExpMax}
Si $e$ est un exposant d'homologie minimal (pour $\leq $) alors $e$ n'appara\^{\i}t que dans $H_0 (\n , V)$.
\end{prop}

Nous utiliserons cette proposition sous la forme suivante~:

\begin{prop} \label{ExpMax2}
Soit $f$ un exposant restreint minimal. Alors $f$ n'appara\^{\i}t que dans $H_0 (\n , V)$.
\end{prop}
\proof Supposons en effet que $f=e_{|\mathfrak{a}^1}$ o\`u $e$ appara\^{\i}t dans $H_q$, $q>0$. Alors $e=e_0 + \sum n_{\alpha} \alpha$ (notation de \ref{section:exp}) o\`u $(n_{\alpha})_{\alpha} \neq 0$ d'apr\`es la proposition \ref{ExpMax} et $e_0$ appara\^{\i}t dans $H_0$. On en d\'eduit que 
$$f = (e_0)_{|\mathfrak{a}^1} + \sum_{\alpha} n_{\alpha} \alpha_{\rm res}.$$
Ce qui, d'apr\`es le lemme \ref{L:res}, contredit la minimalit\'e de $f$.
\qed

Nous avons en vue le th\'eor\`eme suivant -- version tordue de la \og conjecture d'Osborne\fg \ d\'emontr\'ee
par Hecht et Schmid \cite{HechtSchmid}. Suivant la d\'emonstration de \cite{HechtSchmid} nous
commencerons par le r\'eduire \`a un cas apparemment tr\`es particulier {\it via} un proc\'ed\'e de \og continuation coh\'erente \fg.

On d\'efinit~:
\begin{equation} \label{Det}
\begin{split}
D_{\n}^{\theta} (t) & =  \det \left( (1-\theta t)_{|\n} \right) \\
& =  \sum_q (-1)^q {\rm trace} (t\theta \ | \ \Lambda^q \n ) .\\
\end{split}
\end{equation}
On dit que $t \in T^1$ est {\it contractant} si les valeurs propres (complexes) de $t\theta$ dans $\n$ 
sont de valeur absolue $<1$. Noter qu'alors $D_{\n}^{\theta} (t) \neq 0$.\footnote{Nous consid\'erons un 
ensemble un peu plus petit que celui de Hecht-Schmid dans le cas non tordu \cite[\S 3]{HechtSchmid}.} 

\begin{thm} \label{Osborne}
Pour tout \'el\'ement $\theta$-r\'egulier contractant $t\in T^1$, on a~:
$$\Theta_{\pi , \theta} (t) = \frac{\sum_q (-1)^q \Theta_q^{\theta} (t , V)}{D^{\theta}_{\n} (t)}.$$
\end{thm}

\subsection{} On a le fait \'evident suivant~:

{\it Si $t\in T^1$ est contractant et $a \in A^1 =A_{\theta}$ v\'erifie $a^{\alpha_{\rm res}} <1$ ($\alpha_{\rm res} \in R_{\rm res}^{\rm n.i.}$) alors $ta$ est contractant.}

\section{Continuation coh\'erente}

Soit $(V , \pi)$ une repr\'esentation admissible de $G$ munie d'une extension \`a $G^+$. On consid\`ere 
le caract\`ere tordu 
$$\Theta_{\theta} (V) = \Theta_{\pi , \theta}.$$
Il d\'epend \'evidemment du choix de l'op\'erateur d'entrelacement $A_{\theta}$.

De m\^eme chaque $H_q (\n , V)$ est un $MA$-module $\theta$-invariant. Et le caract\`ere
tordu $\Theta_{MA , \theta} (H_q (\n , V))$ est \'egal \`a $\Theta_q^{\theta} (t,V)$. Dans la suite nous
notons~:
$$\Theta_{\n , \theta} (V) = \frac{\sum_q (-1)^q \Theta_{MA , \theta} (H_q (\n , V))}{D^{\theta}_{\n} (t)}.$$

Noter que $V$ est toujours de longueur finie; soit 
$$[V] = m_1 [V_1] + \ldots + m_n [V_n]$$
l'expression de $V_{|G}$ dans le groupe de Grothendieck des $(\g , K)$-modules, avec $V_i$ irr\'eductible, pour $i=1 , \ldots , n$, alors~:
\begin{equation} \label{decIrred1}
\Theta_{\theta} (V) = m_1 ' \Theta_{\theta} (V_1 ) + \ldots + m_n ' \Theta_{\theta} (V_n ) \quad (m_i ' \in \ZZ [\zeta_d] )
\end{equation}
o\`u, par convention, on pose $\Theta_{\theta} (V_i ) = 0$ si $V_i$ n'est pas $\theta$-stable et o\`u on 
choisit, si $V_i$ est $\theta$-stable, une extension \`a $G^+$. 

De m\^eme chaque $H_q (\n , V_i)$ est un $MA$-module $\theta$-invariant. Et le caract\`ere
tordu $\Theta_{\n , \theta} (V)$ 
se d\'ecompose en~:
\begin{equation} \label{decIrred2}
\Theta_{\n , \theta} (V) = m_1 ' \Theta_{\n , \theta} (V_1) + \ldots + m_n ' \Theta_{\n , \theta} (V_n),
\end{equation}
avec les m\^emes coefficients $m_i '$.
Il suffit en effet de traiter le cas d'une suite exacte courte 
$$0 \rightarrow V_1 \rightarrow V \rightarrow V_2 \rightarrow 0$$
o\`u $V_1$ est irr\'eductible en tant que $G^+$-module. L'action de $\theta$ -- par $A_{\theta}$ -- 
sur $V$ induit une action sur $V_1$ et $V_2$ et ces actions commutent \`a la suite exacte longue d'homologie.
Le calcul de la caract\'eristique d'Euler de cette suite exacte longue implique alors l'additivit\'e 
des caract\`eres tordus. Par ailleurs si la restriction du $G^+$-module $V_1$ \`a $G$ n'est pas irreductible $\Theta_{\n , \theta} (V_1) = 0$.    

Il d\'ecoule de \eqref{decIrred1} et \eqref{decIrred2} que pour 
d\'emontrer que pour tout \'el\'ement $\theta$-r\'egulier 
$t\in T^1$, on a~:
\begin{equation} \label{O}
\Theta_{\theta} (V) = \Theta_{\n , \theta} (V),
\end{equation} 
il suffit de le v\'erifier pour les $V$ irr\'eductibles. 

\subsection{} Notons $\mathcal{C}^{\theta}$ le $\ZZ[\zeta_d]$-module libre engendr\'e par les 
caract\`eres tordus de modules d'Harish-Chandra $\theta$-stables admissibles et irr\'eductibles (\'etendus \`a $G^+$). Le module
$\mathcal{C}^{\theta}$ se d\'ecompose en une somme directe~:
$$\mathcal{C}^{\theta} = \bigoplus_{\lambda^1 \in \mathfrak{t}_{\CC}^{1*} / W}  \mathcal{C}^{\theta}_{\lambda^1}, 
\ \ \ \mbox{ avec } \mathcal{C}^{\theta}_{\lambda^1} = \{ \Theta \in \mathcal{C}^{\theta} \; : \; Z(\g_{\CC})  \mbox{ op\`ere sur } \Theta \mbox{ par } \chi_{\lambda} = \chi_{\lambda^1} \circ \mathrm{N} \}.$$
L'anneau $\mathcal{F}^{\theta}$ des repr\'esentations
virtuelles de dimension finie de $\G$ invariante par $\theta$ op\`ere par tensorisation
sur $\mathcal{C}^{\theta}$ et le munit d'une structure naturelle de $\mathcal{F}^{\theta}$-module (cf. d\'emonstration du th\'eor\`eme \ref{T26}). 
Rappelons -- voir \eqref{DF} -- qu'une repr\'esentation $\varphi \in \mathcal{F}^{\theta}$ a un caract\`ere tordu dont la restriction \`a
$T^1$ est de la forme~:
$$\varphi  = \sum_{\mu \in X_p^{\rm res} (\TT^1)} m_{\mu} e^{\mu},$$
o\`u $m_{\mu}$ est la trace de $\theta$ sur l'espace propre de la valeur propre $\mu$. 

\`A tout $\lambda^1 \in \mathfrak{t}_{\CC}^{1*}$ il correspond $\lambda = \mathrm{N} (\lambda^1) \in \mathfrak{t}_{\CC}^{*}$ qui est $\theta$-stable et tel que $\lambda_{| \mathfrak{t}_{\CC}^1} = \lambda^1$. On a alors $\chi_{\lambda} = \chi_{\lambda^1}\circ \mathrm{N}$. Dans la suite on identifie ainsi $\mathfrak{t}_{\CC}^{1 *}$ \`a un sous-espace de $\mathfrak{t}_{\CC}^{*}$. Rappelons en particulier que l'on identifie ainsi les \'elements $\mu \in X_p^{\rm res} (\TT^1)$ \`a
des \'el\'ements $\theta$-stables de $X_p (\TT)$. 

Fixons alors $\lambda_0=\lambda_0^1 \in \mathfrak{t}_{\CC}^{1*}$. On dit qu'une famille 
$$\{ \Theta_{\lambda} \; : \; \lambda \in X_p^{\rm res} (\TT^1 ) + \lambda_0 \} \subset \mathcal{C}^{\theta} ,$$
index\'ee par les translat\'es de $\lambda_0 \in \mathfrak{t}^{*}_{\CC}$ par le r\'eseau
$X_p^{\rm res} (\TT^1 ) \subset X_p (\TT)$, d\'epend de fa\c{c}on {\it coh\'erente} du param\`etre $\lambda$ si 
\begin{equation} \label{338}
\begin{split} 
& \mathrm{(a)} \ \Theta_{\lambda} \in \mathcal{C}^{\theta}_{\lambda^1} \quad \mbox{ et } \\
& \mathrm{(b)} \ \varphi \Theta_{\lambda} = \sum_{\mu \in X_p^{\rm res} (\TT^1)} m_{\mu} \Theta_{\lambda + \mu} ,
\end{split}
\end{equation}
pour tout $\lambda \in X_p^{\rm res} (\TT^1 ) + \lambda_0$ et tout $\varphi \in \mathcal{F}^{\theta}$. Remarquons que si $\lambda = \mu + \lambda_0$ avec
$\mu \in X_p^{\rm res} (\TT^1 )$, alors $\lambda$ est $\theta$-stable de sorte que $\lambda^1 = \lambda_{| \mathfrak{t}_{\CC}^1}$ ou encore $\chi_{\mu + \lambda_0} \circ \mathrm{N} = \chi_{\lambda}$. 

Le choix du tore maximal $T$ n'est pas important dans la d\'efinition ci-dessus. 
Noter que d'apr\`es \S \ref{ChgtAutom} si $\delta \in L$ est un \'el\'ement semisimple qui stabilise
une paire $(\B ' , \T')$ les 
tores $\T^1$ et $\T' \- {}^{1}$ ne diff\`erent que d'un automorphisme int\'erieur de $\G$. 
On peut utiliser cet automorphisme pour transf\'erer la param\'etrisation d'une famille 
coh\'erente \`a une param\'etrisation par des \'el\'ements de $\mathfrak{t}_{\CC}' \- {}^{1*}$. Ainsi
reparam\'etr\'ee, la famille est coh\'erente si et seulement la famille avant reparam\'etrage l'\'etait.

\subsection{} Fixons $\lambda_0 \in \mathfrak{t}_{\CC}^{1*}$.
D'apr\`es le th\'eor\`eme \ref{T26}, on peut \'ecrire~:
$$\left[ |\Delta_{\theta} | \Theta_0 \right] (x \exp H) = \sum_{w \in W} c_0 (x,w) e^{\lambda_0 (wH)}$$
au voisinage d'un \'el\'ement $\theta$-r\'egulier $x\in T^{\theta}$. Ici les constantes $c_0 (x,w)$
ne sont g\'en\'eralement pas uniquement d\'etermin\'ees par $\Theta_0$ mais, si l'on note 
$W_0$ le stabilisateur de $\lambda_0$ dans $W$, la
somme
$$c_{\lambda_0} (x, w) := \sum_{v \in W_0} c_0 (x, wv)$$
est bien uniquement d\'etermin\'ee par $\Theta_0$. Notons $\Theta_{\lambda_0} = N_0 
\Theta_0$, avec $N_0$ \'egal au cardinal de $W_0$. On a alors~:
\begin{equation} \label{init}
\left[ |\Delta_{\theta} | \Theta_{\lambda_0} \right] (x \exp H) = \sum_{w \in W} c_{\lambda_0} (x,w) e^{\lambda_0 (wH)} .
\end{equation}


\begin{lem} \label{L344}
Il existe une famille de caract\`eres tordus 
$$\{ \Theta_{\lambda} \; : \; \lambda \in X_p^{\rm res} (\TT^1 ) + \lambda_0 \}$$ 
qui d\'epende de fa\c{c}on coh\'erente de $\lambda$ et telle que
\begin{enumerate}
\item $\Theta_{\lambda_0} = N_0 \Theta_0$, 
\item $\Theta_{w \lambda} = \Theta_{\lambda}$ ($\lambda \in X_p^{\rm res} (\TT^1 ) + \lambda_0$, $w \in W_0$).
\end{enumerate}
Cette famille v\'erifie en outre les identit\'es
\begin{equation}  \label{idc}
c_{\lambda+\mu} (x, w)  = e^{w^{-1} \mu} (x) c_{\lambda} (x, w)
\end{equation}
pour tous $\lambda \in X_p^{\rm res} (\TT^1 ) + \lambda_0$, $\mu \in X_p^{\rm res} (\TT^1 )$ et $x$
$\theta$-r\'egulier dans $T^{\theta}$.
\end{lem}
\begin{proof} La d\'emonstration est identique \`a celles de \cite[Lem. 3.39 \& 3.44]{HechtSchmid}. Il suffit
de consid\'erer un poids $\mu \in X_p^{\rm res} (\TT^1 )$ suffisamment dominant pour que
$\lambda := \lambda_0 + \mu$ ne soit $W$-conjugu\'e \`a un $\lambda_0 + v \mu$ ($v \in W$)
que si $v \in W_0$. La somme
$$\psi := \sum_{ v \in W} e^{v \mu}$$
d\'efinit un caract\`ere tordu virtuel dans $\mathcal{F}^{\theta}$. Et la conclusion du lemme
nous force \`a prendre $\Theta_{\lambda}$ \'egal \`a la projection de $\psi \Theta_0 \in \mathcal{C}^{\theta}$ dans $\mathcal{C}^{\theta}_{\lambda}$. On v\'erifie facilement l'identit\'e sur les constantes
\eqref{idc} dans ce cas. On se ram\`ene \`a un $\lambda$ quelconque en inversant le proc\'ed\'e. 
\end{proof}

\subsection{} Rappelons que nous abr\'egeons le membre de droite de \eqref{O} en posant~:
\begin{equation} \label{thetan}
\Theta_{\n , \theta} (V) = \frac{\sum_{q} (-1)^q \Theta_{MA , \theta} (H_q (\n , V))}{D^{\theta}_{\n} (t)} .
\end{equation}

Maintenant soit $\lambda \in \mathfrak{t}^*_{\CC}$ tel que $\chi_{\lambda}$ soit \'egal au 
caract\`ere infinit\'esimal de $\pi$. Alors $\chi_{\lambda}$ est $\theta$-stable et il existe $\lambda^1 
\in \mathfrak{t}^{1*}_{\CC}$ tel que $\chi_{\lambda} = \chi_{\lambda^1} \circ N$. 
De m\^eme, un caract\`ere qui intervient de fa\c{c}on non triviale dans la combinaison
lin\'eaire \eqref{HomTordue} est n\'ecessairement $\theta$-stable et il d\'ecoule de 
\cite[Cor. 3.32]{HechtSchmid} qu'il est \'egal \`a $w\lambda^1 + \rho_P$ pour un certain $w \in W$. 
D'apr\`es le th\'eor\`eme \ref{T26} la fonction $|\Delta_{MA , \theta} | \Theta_{MA , \theta} (H_q (\n , V))$ est 
localement une combinaison
lin\'eaire d'exponentielles de la forme $e^{w \lambda^1 + \rho_P}$. Mais on a~:
$$\Delta_{\theta}  = e^{-\rho_P} D_{\n}^{\theta}  \Delta_{MA, \theta} .$$
On en d\'eduit donc que si $x \in T^1$ est $\theta$-r\'egulier, il existe un voisinage $\mathcal{V}$ de $0$ dans $\mathfrak{t}^1$ tel que pour tout $H \in \mathcal{V}$  on a~:
\begin{equation} \label{Fthetan}
\left[ |\Delta_{\theta} | \Theta_{\n , \theta} \right] (x \exp H) = \sum_{w \in W} \tilde{c}_{\lambda} (x,w) 
e^{\lambda^1 (wH)} ,  
\end{equation}
o\`u les $\tilde{c}_{\lambda} (x, w)$ sont des constantes dans $\CC$. Il d\'ecoule alors 
de \cite[Lem. 3.59]{HechtSchmid} que la continuation coh\'erente s'applique
aussi bien \`a $\Theta_{\n , \theta}$ qu'\`a $\Theta_{\theta}$. On peut donc ins\'erer $\Theta_{\n , \theta}$
dans une famille coh\'erente de sorte que 
\begin{equation} \label{360}
\tilde{c}_{\lambda + \mu} (x, w) = e^{w^{-1} \mu} (x) \tilde{c}_{\lambda} (x,w),
\end{equation}
pour tous $\lambda \in X_p^{\rm res} (\TT^1 ) + \lambda_0$, $\mu \in X_p^{\rm res} (\TT^1 )$ et $x$
$\theta$-r\'egulier dans $T^{\theta}$.

D\'emontrer \eqref{O} revient \`a v\'erifier les identit\'es~:
\begin{equation} \label{01}
c_{\lambda} (x, w) = \tilde{c}_{\lambda} (x,w), \ \ \ \mbox{ pour } x \in T^1 \mbox{ $\theta$-r\'egulier}.
\end{equation}
Les relations \eqref{360} et leurs homologues du lemme \ref{L344} impliquent que si $C$ est une constante
strictement positive quelconque, il suffit de v\'erifier \eqref{01} lorsque 
\begin{equation}\label{362}
\mathrm{Re} \langle w^{-1} \lambda , \alpha \rangle < -C \ \ \ \mbox{ pour tout } \alpha \in R_{\rm res}^{\rm n.i.}.
\end{equation}
(Rappelons qu'ici $\lambda = \mathrm{N} (\lambda^1) \in \mathfrak{t}_{\CC}^*$ est $\theta$-stable.)

\section{D\'emonstration du th\'eor\`eme \ref{Osborne}}

Soit donc $V$ le module d'Harish-Chandra d'une repr\'esentation (admissible) irr\'eductible $\theta$-stable $\pi$ (\'etendue \`a $G^+$) de caract\`ere infinit\'esimal $\chi_{\lambda}$ avec $\lambda = \mathrm{N} (\lambda^1) \in \mathfrak{t}_{\CC}^*$. Nous voulons v\'erifer l'identit\'e \eqref{O}.

\subsection{} D'apr\`es le chapitre pr\'ec\'edent (par continuation coh\'erente) on est ramen\'e \`a d\'emontrer \eqref{O} sous
l'hypoth\`ese \eqref{362}. 

Seule la $W$-orbite de $\lambda$ est bien d\'efinie, on peut donc supposer
$\mathrm{Re} \langle \lambda , \alpha \rangle \leq 0$ pour tout $\alpha \in R_{\rm res}^{\rm n.i.}$. De sorte que
\eqref{362} force $\mathrm{Re} \langle \lambda , \alpha \rangle < -C$ et $w$ trivial. 
Il suffit donc de v\'erifier les identit\'es \eqref{01} lorsque~:
\begin{equation} \label{516}
\begin{split}
&\mathrm{(a)} \ {\rm Re} \langle \lambda , \alpha \rangle < -C \ \ \ \mbox{ pour tout } \alpha \in R_{\rm res}^{\rm n.i.}, \mbox{ et } \\
&\mathrm{(b)} \ w \mbox{ est trivial}.
\end{split}
\end{equation}
Ici $C$ est une constante positive \`a fixer ult\'erieurement. 

Puisque $\lambda = \mathrm{N} (\lambda^1) \in \mathfrak{t}_{\CC}^*$ on a $\lambda^1 = \lambda_{| \mathfrak{t}_{\CC}^1}$. Posons $e= \lambda_{| \mathfrak{a}}$ et $f = \lambda_{| \mathfrak{a}^1}$.

On d\'efinit la $f$-composante 
de $\Theta_{\n , \theta} (V)$, au voisinage d'un \'el\'ement $\theta$-r\'egulier $x \in T^1$, par~:
\begin{equation} \label{ethetan}
\begin{split}
\Theta_{\n , \theta} (V)_{f} & = \frac{\sum_{w} \tilde{c}_{\lambda} (x,w) 
e^{\lambda (wH)}}{|\Delta_{MA, \theta}|} e^{\rho_P (H)}  \\
& = \sum_q (-1)^q \Theta_{MA, \theta} (H_q (\n , V)_f ), 
\end{split}
\end{equation}
o\`u dans la premi\`ere somme $w$ parcourt l'ensemble 
$$\{ w \in W \; : \; \lambda (w \cdot)_{| \mathfrak{a}} = e\}=\{ w \in W \; : \; \lambda (w \cdot)_{| \mathfrak{a}^1} = f\}.$$ 
De la m\^eme mani\`ere on d\'efinit
\begin{equation} \label{etheta}
\Theta_{\theta} (V)_f = \frac{\sum_{w} c_{\lambda} (x,w) e^{\lambda (wH)}}{|\Delta_{MA, \theta}|} e^{\rho_P (H)}.
\end{equation}
D'apr\`es \eqref{516}, le th\'eor\`eme \ref{Osborne} d\'ecoulera de la v\'erification que~:
\begin{equation} \label{521}
\Theta_{\theta} (V)_f = \Theta_{\n , \theta} (V)_f.
\end{equation}

\begin{lem} \label{Ltruc}
Supposons que $\lambda$ v\'erifie \eqref{516} avec $C=0$.
Alors sur les \'el\'ements $\theta$-r\'eguliers de $T^1$, on a~:
$$\Theta_{\n , \theta} (V)_{f} = \Theta_{MA,\theta} (H_0 (\n , V)_{f} ).$$
\end{lem}
\begin{proof}
Tous les $\theta$-exposants d'homologie apparaissant dans le membre de droite de \eqref{Fthetan} sont de la forme $w\lambda _{| \mathfrak{a}^1}$. 
Mais pour tout $w \in W$, le caract\`ere $w\lambda$ est \'egal \`a $w \lambda_0 + w \mu$ et il d\'ecoule donc
de \eqref{516} (en prenant $C=0$), que l'on peut \'ecrire~: 
$$w \lambda = \lambda + \sum_{\alpha \in R_{\rm res}^{\rm n.i.}} x_{\alpha} \alpha \ \ \ (x_{\alpha} \geq 0).$$ 
En particulier $\lambda_{| \mathfrak{a}^1}$ ne peut s'\'ecrire $\lambda_{| \mathfrak{a}^1} = 
w\lambda _{| \mathfrak{a}^1} + \sum m_{\alpha} \alpha_{\rm res}$, comme on le voit par restriction de 
l'\'egalit\'e ci-dessus \`a $\mathfrak{a}^1$. L'exposant restreint $f$ est donc minimal et n'appara\^{\i}t que dans $H_0$, d'apr\`es la 
proposition \ref{ExpMax2}.
\end{proof}

Remarquons que puisque $f= e_{| \mathfrak{a}^1}$ avec $e$ $\theta$-stable, on a~:
\begin{equation} \label{ef}
\Theta_{MA,\theta} (H_0 (\n , V)_{f} ) = \Theta_{MA,\theta} (H_0 (\n , V)_{e} ).
\end{equation} 
Ainsi bien qu'en g\'en\'eral $H_0 (\n , V )_e =0$ n'implique pas $H_0 (\n , V)_f = 0$, cela
implique n\'eanmoins que $\Theta_{MA,\theta} (H_0 (\n , V)_{f} ) = 0$. La proposition suivante montre que l'\'enonc\'e correspondant est \'egalement vrai pour $\Theta_{\theta} (V)_{f}$.

\begin{prop} \label{615}
Supposons que $\lambda$ v\'erifie \eqref{516} avec $C=0$ et que $H_0 (\n , V)_e = 0$. Alors $\Theta_{\theta} (V)_f$ est nul sur tout \'el\'ement $\theta$-r\'egulier contractant de $T^1$.
\end{prop}

La d\'emonstration de cette proposition occupe la fin de ce chapitre. Commen\c{c}ons par
montrer comment en d\'eduire \eqref{521}. C'est imm\'ediat si $H_0 (\n , V)_e = 0$~: \`a la fois 
$\Theta_{\theta} (V)_f$ et $\Theta_{\n , \theta} (V)_f$ sont nuls sur les \'el\'ements $\theta$-r\'eguliers 
contractants de $T^1$. Supposons donc $H_0 (\n , V)_e \neq 0$. 

\begin{lem} \label{615b}
Supposons $C$ assez grand. Alors $e$ est en position de 
Langlands.
\end{lem}
\begin{proof} Puisque $e=\lambda_{| \mathfrak{a}}$ est $\theta$-stable, cela d\'ecoule de \eqref{516}.
\end{proof}

On d\'eduit alors de Hecht-Schmid \cite[Lem. 6.10]{HechtSchmid} que (pour $C$ assez grand), le $M$-module $W=H_0 (\n , V)_e$ est irr\'eductible.


Soit alors 
$$I={\rm ind}_{P}^{G}  (W \otimes \CC_{e})$$ 
(induction unitaire depuis $MAN$): c'est un $G^+$-module de fa\c{c}on naturelle, et on a un morphisme naturel $V \hookrightarrow I$. Le sous-quotient $V$ de $I$ est le seul
qui v\'erifie $H_0 (\n , V)_e  \neq 0$ (cf. Hecht-Schmid, d\'emonstration du lemme 6.10). Donc 
$$\Theta_{\theta} (V)_f = \Theta_{\theta} (I)_f$$
\`a l'aide, de nouveau, de la proposition \ref{615}, sur les \'el\'ements $\theta$-r\'eguliers contractants 
de $T^1$. De plus (avec la d\'efinition, contenant la translation par $\rho_P$, de $H_0 (\n , V)_e$), on a
$$H_0 (\n , V)_e \cong W \otimes \CC_{e+ \rho_P}$$
comme $MA$-module. L'identit\'e \eqref{521} d\'ecoule donc du lemme \ref{Ltruc}, de l'\'equation \eqref{ef} et du 
calcul suivant relatif \`a la repr\'esentation induite: 

\begin{lem} \label{517}
$$\Theta_{\theta} (I)_f = \Theta_{MA , \theta} (W \otimes \CC_{e+\rho_P}).$$
\end{lem}
\begin{proof}
Noter que $Z(\mathfrak{m}_{\CC} \oplus \mathfrak{a}_{\CC} )$ op\`ere sur $W \otimes \CC_e$ selon le 
caract\`ere $\chi_{\lambda}$. Partons donc de l'expression locale du caract\`ere tordu 
$\Theta_{\theta} (I)$ donn\'ee par \eqref{514}~:
$$
\left[ |\Delta_{\theta} | \Theta_{I ,\theta}\right]  = \frac{1}{\# W_{\theta} (MA , T)} \sum_{v \in W_{\theta} (G , T)} \sum_{w \in W(\mathfrak{m}_{\CC} \oplus \mathfrak{a}_{\CC} , \mathfrak{t}_{\CC})^{\theta}} d_{\lambda} (vx , w) e^{v^{-1} w^{-1} \lambda}.$$
Un terme de la double somme contribue \`a $\Theta_{\theta} (I)_e$ si et seulement si 
$(v^{-1} w^{-1} \lambda )_{|\mathfrak{a}^1} = f$ (avec 
$v \in W_{\theta} ( G, T)$ et $w \in W (\mathfrak{m}_{\CC} \oplus \mathfrak{a}_{\CC} , \mathfrak{t}_{\CC}
)^{\theta}$). Mais, en prenant $C=0$ dans \eqref{516}, l'\'egalit\'e
$$(v^{-1} w^{-1} \lambda )_{|\mathfrak{a}^1} = f = \lambda_{|\mathfrak{a}^1}$$
force $v$ \`a appartenir au groupe $W_{\theta} (MA , T)$. 
Donc seuls les termes avec $v \in W_{\theta} (MA , T)$ dans la double somme contribuent \`a $\Theta_{\theta} (I)_f$. 

Maintenant, le groupe $W_{\theta} (MA , T)$ op\`ere trivialement sur le caract\`ere tordu de 
$W \otimes \CC_e$ . On a donc~:
$$d_{\lambda} (v x , w) e^{v^{-1} w^{-1} \lambda } = d_{\lambda} (x,w) e^{w^{-1} \lambda}$$
pour tout $v \in W_{\theta} (MA , T)$. On obtient donc que l'expression locale de $|\Delta_{MA , \theta}|\Theta_{\theta} (I)_f$
est \'egale \`a
$$\sum_{w \in W (\mathfrak{m}_{\CC} \oplus \mathfrak{a}_{\CC} , \mathfrak{t}_{\CC}
)^{\theta}} d_{\lambda} (x,w) e^{w^{-1} \lambda + \rho_P}$$
qui est aussi l'expression locale de $|\Delta_{MA , \theta} | \Theta_{MA , \theta} (W \otimes \CC_{e+\rho_P})$.
\end{proof}

Dans les paragraphes qui suivent on met progressivement en place les ingr\'edients n\'ecessaires 
\`a la d\'emonstration de la proposition \ref{615}. On suppose dor\'enavant que $H_0 (\n , V)_e = 0$. Nous suivons la d\'emonstration de \cite[Prop. 6.15]{HechtSchmid}.

\subsection{} 
D'apr\`es \cite[(4.24)]{HechtSchmid}, tout exposant 
dominant $e'$ du module de Harish-Chandra $V$
est de la forme $e'=(w \cdot \lambda)_{|\mathfrak{a}}$ pour un certain $w \in W(\G , \T)$ et intervient 
dans $H_0 (\n , V)$. Puisque l'on suppose $H_0 (\n , V)_e$ \'egal \`a $\{ 0 \}$, $e$ n'est pas un 
exposant dominant, l'exposant $e'$ est donc diff\'erent de $e$. Soit il est $\theta$-stable et 
$e'_{| \mathfrak{a}^1} \neq f$ soit il n'est pas $\theta$-stable et on \'ecrit 
$e'  = we = e + \sum_{\alpha \in R^{\rm n.i.}} x_{\alpha} \alpha$ o\`u les $x_{\alpha}$ sont positifs (car, d'apr\`es \eqref{516}, $e$ est tr\`es n\'egatif) et non tous nuls. Soit 
$\mathrm{N}$ la projection sur $\mathfrak{a}_1$~: puisque $\mathrm{N} e = e$, on a $\mathrm{N} e' = e + \sum_{\alpha} x_{\alpha} \mathrm{N} \alpha$. 
Si $\mathrm{N} e'$ est diff\'erent de $e$, il existe donc un $x_{\alpha} >0$ tel que $\mathrm{N} \alpha \neq 0$.
Dans tous les cas la condition \eqref{516} (pour $C$ assez grand) implique qu'il existe une racine $\alpha_{\rm res} \in R_{\rm res}^{\rm n.i.}$ et une constante $\delta$ strictement positive telles que
pour tout $\beta \in R_{\rm res}^{\rm n.i.}$ on a~:
\begin{equation} \label{656}
{\rm Re} \langle e'_{|\mathfrak{a}^1} , \beta \rangle \geq {\rm Re} \langle f , \beta \rangle + \delta \langle \alpha_{\rm res}, \beta \rangle \ \mbox{ et } \ \alpha_{\rm res} \- {}_{| \mathfrak{a}^1} \neq 0 .
\end{equation}
Soit $A^{1-} = \exp \mathfrak{a}^{1-}$.
Quitte \`a diminuer $\delta$, il d\'ecoule de \eqref{656} et de \cite[Thm. 8.47]{Knapp} (voir \'egalement \cite[Lem. 6.46]{HechtSchmid})~\footnote{On prendra garde au fait que les notions d'exposants diff\`erent d'une
translation par $\rho_P$ dans ces deux r\'ef\'erences ainsi qu'au fait que dans l'une l'asymptotique est
dans $A^+$ dans l'autre dans $A^-$. Nous suivons les conventions de \cite{HechtSchmid}.}
que tout coefficient
matriciel $K$-fini de $\pi$ est major\'e sur $\overline{A^{1-}}$ par un multiple de 
$e^{({\rm Re} f + \delta \alpha_{\rm res} +\rho_P )} (a)$.

\subsection{} Les r\'esultats du paragraphe pr\'ec\'edent se traduisent en une borne sur $\Theta_{\pi , \theta}$.

Pour tout $\tau \in \widehat{K}$ on note $d_{\tau}$ le degr\'e de $\tau$. Soit
$E_{\tau}$ la projection orthogonale de $V$ sur $V_{\tau}$ la composante
$\tau$-isotypique de $V$. En choisissant une base orthonormale de $V$ compatible avec la 
d\'ecomposition de $V$ en $K$-types, on voit que
\begin{equation} \label{sumdist}
\Theta_{\pi , \theta} (g) = \sum_{\tau \in \widehat{K}} {\rm trace} (E_{\tau} \pi (g) A_{\theta} E_{\tau})
\end{equation}
au sens des distributions.
Noter que le sous-groupe compact $K$ est $\theta$-stable et donc que $A_{\theta}$ pr\'eserve la 
d\'ecomposition de $V$ en $K$-types. 

Maintenant ${\rm trace} (E_{\tau} \pi (g) A_{\theta} E_{\tau})$ est une somme de coefficients matriciels
associ\'es \`a des vecteurs de $V_{\tau}$. Un proc\'ed\'e bien connu permet de se ramener
\`a des coefficients matriciels $K$-finis, voir par exemple \cite[Lem. 6.23]{HechtSchmid} ou \cite[Lem. 3.4.1 \& Lem. 3.5.1]{Bouaziz} (pour les groupes
non-connexes). Il d\'ecoule alors du paragraphe pr\'ec\'edent que pour tout $a \in A^{1-}$ on a~:
\begin{equation} \label{chh}
|{\rm trace} (E_{\tau} \pi (a) A_{\theta} E_{\tau})| = O\left( d_{\tau}^3 e^{({\rm Re} f + \delta \alpha_{\rm res} + \rho_P ) } (a) \right).
\end{equation}

\subsection{} \label{xjhj} \'Etant donn\'e un r\'eel strictement positif $\varepsilon$ on pose 
$$T^{1}_{\varepsilon} = \left\{ t \in T^1 \; : \;  |\lambda_{\alpha_{\rm res}  , j} e^{\alpha_{\rm res}} (t) -1| > \varepsilon, \ \forall \alpha_{\rm res} \in R_{\rm res}^{\rm n.i.}, \ j = 1 , \ldots , n_{\alpha_{\rm res}} \right\}.$$
(Les notations sont celles du \S \ref{1.4}.)
Noter que $T^{1}_{\varepsilon}$ est ouvert dans $T^1$, et que la r\'eunion $\cup_{\varepsilon >0} T^{1}_{\varepsilon}$
co\"{\i}ncide avec l'ensemble des \'el\'ements $\theta$-r\'eguliers contractants de $T^1$.
En utilisant les r\'esultats du paragraphe pr\'ec\'edent, 
la d\'emonstration de \cite[Lem. 6.39]{HechtSchmid} (voir \'egalement \cite[Lem. 3.5.1]{Bouaziz} dans le cas
non-connexe) implique que la distribution qui \`a une fonction $g$ \`a support compact
dans l'ensemble des \'el\'ements $\theta$-r\'eguliers associe l'int\'egrale
$$\int_{T^1}  | \Delta_{\theta}| \Theta_{\theta , \pi} g dt $$
s'exprime comme une combinaison lin\'eaire $\sum_{j=1}^m X_j h_j$ o\`u les $X_j$ sont
des op\'erateurs diff\'erentiels invariants sur $T^1$ et les $h_j$ des fonctions continues sur 
les \'el\'ements $\theta$-r\'eguliers contractants de $T^1$ telles que
\begin{equation} \label{hj}
\begin{split}
|h_j (t a ) | & = O_{\varepsilon} \left( |\Delta_{\theta} (ta) | e^{({\rm Re} f + \delta \alpha_{\rm res} + \rho_P ) } (a) \right) \\
& = O_{\varepsilon} \left( e^{({\rm Re} f + \delta \alpha_{\rm res}) } (a) \right), \ \ \ (ta \in T^1_{\varepsilon} , \ a \in A^{1-} ).
\end{split}
\end{equation}
Ici les $X_j$ et les $h_j$ peuvent \^etre choisis ind\'ependamment de $\varepsilon$.

\subsection{D\'emonstration de la proposition \ref{615}}

Supposons par l'absurde que $\Theta_{\theta} (V)_f$ soit non nul au voisinage d'un \'el\'ement
$\theta$-r\'egulier contractant $t \in T^1$. Il existe alors un \'el\'ement $w \in W$ tel que 
\begin{equation} \label{668}
c_{\lambda} (t , w^{-1}) \neq 0 \ \ \ \mbox{ et } \ \ \ (w \cdot \lambda^1 )_{| \mathfrak{a}^1} = f.
\end{equation}
On fixe un voisinage compact $U$ de l'identit\'e dans $K \cap T^1$, suffisamment petit 
pour que l'on ait~:
$$tU A^{1-} \subset T_{\varepsilon}^1.$$

On numerote $\mu_1 = f , \ldots , \mu_n \in \mathfrak{a}^{1 *} \otimes {\CC}$ les \'el\'ements de 
$\{ \mu_{|\mathfrak{a}^1} \; : \; \mu \in W \cdot \lambda^1 \}$
sans r\'ep\'etition. L'hypoth\`ese \eqref{516} entra\^{\i}ne que 
$$|e^{\mu_i - \mu} (a) | <1, \ \ \ \mbox{ pour } i=2 , \ldots , n, \  a \in A^{1-}.$$
D'apr\`es \eqref{3628} il existe des fonctions $\varphi_1 , \ldots , \varphi_n$, de classe $C^{\infty}$ 
sur un voisinage de $U$ dans $T^1_c$, telles que
\begin{equation} \label{674}
\left[ |\Delta_{\theta} | \Theta_{\theta , \pi} \right] (tma) = \sum_{i=1}^n \varphi_i (m) e^{\mu_i} (a) \ \ \ (m\in U , \ a \in A^{1-} ).
\end{equation} 
L'hypoth\`ese \eqref{668} implique en outre que $\varphi_1$ n'est pas constante \'egale \`a $0$ sur $U$. 
Il existe donc une fonction $\psi$ de classe $C^{\infty}$ sur $T_c^1$, et de support dans
$U$, telle que
$$\int_{T_c^1} \varphi_1 \psi dm =1.$$
Pour $a\in A^{1-}$, on pose
$$\Psi (a) = e^{-\mu} (a) \int_{m \in T_c^1} \left[ |\Delta_{\theta} | \Theta_{\theta , \pi} \right] (tma) \psi (m)
dm .$$
D'apr\`es \eqref{674}, on peut \'ecrire $\Psi$ comme une somme d'exponentielles~:
\begin{equation} \label{676}
\Psi = 1 + \sum_{i=2}^{n} c_i e^{\mu_i - \mu}.
\end{equation}
D'un autre c\^ot\'e, il d\'ecoule du \S \ref{xjhj} qu'il existe des op\'erateurs diff\'erentiels invariants
$X_1, \ldots , X_m$ sur $A^1$, des fonctions continues $h_1 , \ldots , h_m$ sur $A^{1-}$, une forme
lin\'eaire $\tau \in \mathfrak{a}^{1*}$ et des constantes $C_1 , \ldots , C_m$ tels que 
\begin{equation} \label{677}
\begin{split}
\Psi = \sum_{j=1}^{m} X_j h_j , & \mbox{ au sens des distributions}, \ |h_j | \leq C_j e^{\tau} \ (i=1 , \ldots , m), \\
& \mbox{et } \quad e^{\tau} (a) <1 \mbox{ pour } a \in A^{1-}.
\end{split}
\end{equation}
Soit $g$ une fonction $C^{\infty}$ \`a support compact dans $A^{1-}$ telle que 
$\smallint_{A^1} g da =1$ et soient $g_1 , g_2 , \ldots$ les translat\'es de $g$ par une suite de points
$a_1 , a_2 , \ldots$ dont les inverses tendent vers l'infini dans $A^{1-}$. Puisque les exponentielles
$e^{\mu_i -\mu}$ ($i=2 , \ldots , n$) d\'ecroissent le long de $A^{1-}$, il d\'ecoule de \eqref{676} que
$\smallint_{A^1} g_k \Psi da$ tend vers $\smallint_{A^1} f da=1$ lorsque $k$ tend vers l'infini. 
De m\^eme \eqref{677} implique~:
$$\int_{A^1} g_k \Psi da = \sum_{j=1}^m \int_{A^1} (X_j^* g_k) h_j da \rightarrow 0.$$
Ce qui fournit la contradiction recherch\'ee.

\section{Paquets d'Arthur pour les groupes complexes}

Soit $N$ un entier sup\'erieur ou \'egal \`a $1$. Consid\'erons un groupe 
$H$ \'egal \`a $\SO (2 \ell +1 , \CC)$ ou $\SO (2\ell , \CC)$ si $N = 2\ell$ est pair et \'egal \`a $\Sp (2\ell )$  si $N=2\ell +1$ est impair. Rappelons que  $H$ est un sous-groupe endoscopique de $G^+$
o\`u $G = \GL (N , \CC)$ et $\theta$ est l'automorphisme involutif $g \mapsto J {}^t \- g^{-1} J^{-1}$. Le
groupe dual $\widehat{H}$ est \'egal \`a $\Sp (N , \CC)$ ou $\SO (N , \CC)$, naturellement plong\'e dans 
$\GL (N , \CC)$.

Soit 
$$\psi :  \CC^* \times \SL (2 , \CC)  \rightarrow \widehat{H}$$
un param\`etre d'Arthur pour $H$. On suppose de plus que $\psi_{| \CC^*}$ est {\it temp\'er\'ee}, {\it i.e.}
unitaire. D'apr\`es la conjecture de Ramanujan, cela devrait \^etre toujours le cas. Il n'est, en outre, 
pas difficile d'\'etendre nos r\'esultats aux param\`etres \og g\'en\'eralis\'es \fg \ o\`u $\psi_{| \CC^*}$
n'est plus n\'ecessairement unitaire mais contr\^ol\'e par l'approximation de la conjecture de Ramanujan
d\'emontr\'ee par Luo, Rudnick et Sarnak.

\subsection{} Rappelons la d\'efinition du paquet $\small\prod (\psi)$. 
Il y a une notion naturelle de fonctions associ\'ees $(\varphi , f )$, $\varphi \in C_c^{\infty} (G)$, 
$f \in C_c^{\infty} (H)$.~\footnote{On peut en fait consid\'erer seulement des fonctions {\it $K$-finies} des deux c\^ot\'es.} Lorsque $N$ est pair et $H=\SO (N , \CC)$ il faut en outre supposer $f$ invariante par 
un automorphisme ext\'erieur $\alpha$ de $H$; on suppose $\alpha^2 =1$. 

On \'ecrit le param\`etre $\psi$ sous la forme~:
$$\psi = \chi_1 \otimes R_{a_1} \oplus \ldots \oplus \chi_m \otimes R_{a_m} \subset \GL (N , \CC),$$
o\`u chaque $\chi_j$ est un caract\`ere unitaire de $\CC^*$ que l'on \'ecrit 
$z \mapsto z^{p_j} \bar z ^{q_j}$ avec $\mathrm{Re} (p_j + q_j) = 0$. 
Noter que puisque l'image de $\psi$ est contenue dans $\widehat{H}$, le
param\`etre $\psi$ est $\theta$-stable. On en d\'eduit que soit $\chi_j$ est quadratique, soit il existe
$k$ telle que $a_k = a_j$ et $\chi_k = \chi_j^{-1}$. On associe au param\`etre $\psi$ la 
repr\'esentation de $\GL(N , \CC)$~: 
\begin{equation} \label{Pi}
\Pi = \Pi_{\psi} =  \mathrm{ind} (  \chi_1 \circ \det \otimes \ldots \otimes \chi_m \circ \det )
\end{equation}
(induction unitaire \`a partir du parabolique $(a_1 , \ldots , a_m)$). Elle est irr\'eductible d'apr\`es Vogan \cite{VoganU} ou Bernstein \cite{Bernstein} et Baruch \cite{Baruch}. Elle est par ailleurs $\theta$-stable; 
fixons $A_{\theta} : \Pi \rightarrow \Pi$ un entrelacement tel que $A_{\theta}^2 = 1$. 

(Lorsque $N$ est pair et $H=\SO (N , \CC)$ on identifie deux repr\'esentations irr\'eductibles de $G$
conjugu\'ees par $\alpha$. On d\'esigne simplement par $\pi$ la classe d'\'equivalence. Alors 
$\mathrm{trace} \ \pi (f)$ est bien d\'efinie pour $f$ restreinte comme expliqu\'e ci-dessus.)

Le r\'esultat suivant est alors le th\'eor\`eme 30.1 d'Arthur \cite[\S 30]{Arthur}.

\begin{thm}[Arthur] \label{thm:Arthur}
Il existe une famille finie $\small\prod (\psi)$ de repr\'esentations de $H$, et des multiplicit\'es $m(\pi) >0$ ($\pi \in \small\prod (\psi)$) telles que, pour $\varphi$ et $f$ associ\'ees~:
\begin{equation} \label{TrIdent}
\mathrm{trace} \left( \Pi (\varphi ) A_{\theta} \right) = \sum_{\pi \in \small\prod (\psi)} \varepsilon (\pi ) m (\pi) \mathrm{trace} \pi (f),
\end{equation}
o\`u chaque $\varepsilon (\pi)$ est un signe $\in \{ \pm 1\}$. 
\end{thm}

On pourrait aussi -- comme Arthur le fait -- d\'efinir $\small\prod (\psi)$ comme un ensemble de repr\'esentations-avec-multiplicit\'es. L'\'egalit\'e \eqref{TrIdent} d\'etermine uniquement cet ensemble ainsi 
que les signes $\varepsilon (\pi)$. Arthur d\'etermine d'ailleurs ces signes pour un choix convenable
de $A_{\theta}$ (normalisation par le mod\`ele de Whittaker).  Il montre en particulier que $\varepsilon (\pi)$ est constant sur le paquet 
$\small\prod (\psi)$ d\`es que l'\'el\'ement 
$$s_{\psi} =  \psi \left( 1 , \left(
\begin{array}{cc}
-1 & 0 \\
0 & -1 
\end{array} \right) \right)$$
est central dans $\widehat{G}$, voir \cite[p. 246, p. 242]{Arthur}. Noter que ceci se produit en particulier
si pour tout $j= 1 , \ldots , m$, les entiers $a_j$ sont de m\^eme parit\'e. En g\'en\'eral on associe \`a 
$s_{\psi} \in \widehat{H}$ le groupe (complexe) $E$ de groupe dual $\widehat{E}$ \'egal au centralisateur de $s_{\psi}$ dans $\widehat{H}$. Le groupe $E$ appartient \`a $\mathcal{E}_{\rm ell} (H)$ -- l'ensemble des sous-groupes endoscopiques 
elliptiques de $H$. 

Si $N=2\ell$ et $H= \SO(2\ell +1 , \CC)$, l'ensemble $\mathcal{E}_{\rm ell} (H)$ est param\'etr\'e par les couples d'entiers pairs $(N' , N'')$ avec $N'' \geq N' \geq 0$ et $N=N'+N''$. 
Le groupe endoscopique correspondant est le groupe 
$$H' \times H'' = \SO (N'+1 , \CC ) \times \SO (N'' +1 , \CC).$$
Notons alors que $\widehat{H} ' \times \widehat{H}''= \Sp (N' , \CC ) \times \Sp (N'' , \CC)$. Dans ce cas on pose
$$G' \times G'' = \GL (N' , \CC ) \times \GL (N'' , \CC).$$

Si $N=2\ell$ et $H= \SO(2\ell , \CC)$, l'ensemble $\mathcal{E}_{\rm ell} (H)$ est param\'etr\'e par les couples d'entiers pairs $(N' , N'')$ avec $N'' \geq N' \geq 0$ et $N=N'+N''$. 
Le groupe endoscopique correspondant est le groupe 
$$H' \times H'' = \SO (N', \CC ) \times \SO (N'' , \CC).$$
Notons alors que $\widehat{H} ' = \SO (N' , \CC ) \times \SO (N'' , \CC)$. Dans ce cas on pose encore
$$G' \times G''= \GL (N' , \CC ) \times \GL (N'' , \CC).$$

Si $N=2\ell+1$ et $H= \Sp(2\ell  , \CC)$, l'ensemble $\mathcal{E}_{\rm ell} (H)$ est param\'etr\'e par les couples d'entiers pairs $(N' , N'')$ avec $N'' , N' \geq 0$ et $N=N'+(N''+1)$. 
Le groupe endoscopique correspondant est le groupe 
$$H' \times H''= \SO (N' , \CC ) \times \Sp (N''  , \CC).$$
Notons alors que $\widehat{H} ' \times \widehat{H}'' = \SO (N' , \CC ) \times \SO (N''+1 , \CC)$. Dans ce cas on pose
$$G' \times G''= \GL (N' , \CC ) \times \GL (N'' +1 , \CC).$$

Dans tous les cas la param\`etre $\psi$ se factorise \`a travers $\widehat{H}'\times \widehat{H}''$ et on note 
$$(\psi' , \psi '') : \CC^* \times \SL (2 , \CC) \rightarrow \widehat{H}' \times \widehat{H}''$$ 
le param\`etre correspondant.

Il y a encore une notion naturelle de fonctions associ\'ees~:
$$\begin{array}{ccc}
\varphi \in C_c^{\infty} (G) & \longleftrightarrow & f \in C_c^{\infty} (H) \\
\updownarrow & & \updownarrow \\
(\varphi' , \varphi '') \in C_c^{\infty} (G' \times G'') & \longleftrightarrow & (f ' , f'' ) \in C_c^{\infty} (H' \times H'')
\end{array}
$$

Soit $\Pi' \otimes \Pi ''= \Pi_{\psi '} \otimes \Pi_{\psi ''}$ la repr\'esentation de $G' \times G''$ associ\'ee au param\`etre $(\psi ' , \psi'')$. Le th\'eor\`eme 30.1 d'Arthur \cite[\S 30]{Arthur} implique~:
\begin{equation} \label{TrIdent2}
\mathrm{trace} \left( (\Pi' (\varphi' ) \otimes \Pi '' (\varphi '')) A_{\theta} \right) = \varepsilon \sum_{\pi \in \small\prod (\psi)} m (\pi) \mathrm{trace} \ \pi (f),
\end{equation}
o\`u $\varepsilon$ est un signe qu'Arthur d\'etermine explicitement pour un choix convenable de $A_{\theta}$. 

\subsection{} \label{par:6.3} Soit $T_H=(\CC^* )^{\ell}$ le tore maximal diagonal dans $H$. 
On a $T_H=T_{H,c} A_H$, o\`u $A_H =(\RR^*_+ )^{\ell}$ est d\'eploy\'e sur $\RR$ et $T_{H,c}$ est compact. 
Un {\it exposant} de $H$
est un caract\`ere de $A_H$, c'est-\`a-dire un \'el\'ement de $\mathrm{Hom}(\RR^{\ell} , \CC) = \CC^{\ell}$.

La paire usuelle $(\B , \T)$ de $\G$ d\'etermine un syst\`eme de racines positives de $(T_H , H)$. 
On \'ecrit, pour deux exposants $e$, $e'$ de $H$, $e\leq_H e'$ si 
$$e' = e + \sum_{\alpha} n_{\alpha} \alpha \ \ \  (n_{\alpha} \geq 0 )$$
o\`u $\alpha$ d\'ecrit les racines simples et $e$ et $e'$ sont vus comme des formes lin\'eaires complexes
sur $\RR^{\ell}$. Lorsque $H = \SO (2\ell +1 , \CC)$ ou $\Sp (2\ell , \CC)$ les ordres ainsi obtenus
sont d\'ecrits dans l'exemple \ref{section:exp}. 
Lorsque $H = \SO (2\ell , \CC)$ un calcul simple donne que pour deux exposants $e$, $e'$ l'ordre 
$\leq_H$ s'exprime par $e'-e= (x_i)_{i=1 , \ldots , \ell}$ avec 
\begin{equation} \label{ordre2}
\left\{
\begin{array}{l}
x_1 + \ldots + x_i \in \N \ \ \ (1 \leq i \leq \ell -2)  \\
x_1 + \ldots + x_{\ell} \in 2\N, \\
x_1 + \ldots + x_{\ell-1} - x_{\ell} \in 2\N.
\end{array}
\right.
\end{equation}

\subsection{} Lorsque $H = \SO (N+1 , \CC)$, si $N$ est pair, et 
$H= \Sp (N-1, \CC)$, si $N$ est impair, on a introduit une application naturelle bijective $\mathcal{A}_{H/G}$ -- not\'ee $\mathcal{A}$ dans l'exemple \ref{sec:Norme} -- des classes de conjugaison
dans $H$ vers les classes de $\theta$-conjugaison dans $G$.

Lorsque $H= \SO (N, \CC)$, avec $N$ pair, les classes de conjugaison semi-simples de $H$ sont repr\'esent\'ees par 
\begin{equation}
t  = \mathrm{diag} (x_1 , \ldots , x_{\ell} , x_{\ell}^{-1} , \ldots , x_{1}^{-1}) \in T_H
\end{equation}
modulo le groupe de Weyl de $W_H = \mathfrak{S}_{\ell} \rtimes \{\pm 1 \}^{\ell-1}$, o\`u $\{-1\}^{\ell -1}$ est le sous-groupe de $\{-1\}^{\ell}$ d\'efini par $\prod s_i =1$.
Les classes de $\theta$-conjugaison semi-simples de $G$ sont quant \`a elles repr\'esent\'ees 
par les \'el\'ements 
\begin{equation}
\tilde{t} = \mathrm{diag} (s , 1 )  \in T  \  (s \in (\CC^* )^{\ell })
\end{equation}
modulo $W$. On a alors $\tilde{t} = \mathcal{A}_{H/G} (t)$ si $s= (x_1 , \ldots , x_{\ell})$. En sens 
inverse on \'ecrira $t = \mathcal{N}_{H/G} (\tilde{t})$. Noter que dans ce cas (o\`u $H=\SO(2\ell , \CC)$) 
un \'el\'ement $\theta$-r\'egulier $\tilde{t}$ de $G$ a deux ant\'ec\'edents par $\mathcal{A}_{H/G}$.

Par d\'efinition un $\theta$-exposant de $G$ est un caract\`ere du tore maximal d\'eploy\'e $A$. 
Il d\'efinit un caract\`ere de $A_H$ par composition avec l'application norme. 

On v\'erifie alors aussit\^ot \`a l'aide de \eqref{ordre}, \eqref{ordre3} et \eqref{ordre2} que si $e$ et $e'$ sont deux 
$\theta$-exposants de $G$ tels que $e'\leq_H e$ alors $e' \leq_{\theta} e$.

Dans cette section, nous d\'emontrons la proposition suivante.

\begin{prop} \label{prop:exp}
Soit $\pi$ une repr\'esentation arbitraire de $\small\prod (\psi )$ et $e$ un exposant minimal de $\pi$. 
Alors, il existe un $\theta$-exposant d'homologie $e'$ (resp. $e''$) de $\Pi '$ (resp. $\Pi ''$) tel que 
$$e' + e'' \leq_{H} e.$$
\end{prop}
On a identifi\'e le tore $A_H$ au produit $A_{H'} \times A_{H''}$.

\'Etant donn\'e une repr\'esentation $\pi$ de $H$, de longueur finie, on note 
$V$ son module d'Harish-Chandra. Notons par ailleurs $\n_H$ l'alg\`ebre de 
Lie (complexifi\'ee) du sous-groupe unipotent maximal de $H$ donn\'e par $T_H$ et notre choix de racines, et
notons $\rho_H$ la demi-somme des racines usuelles.

Soit $\Theta_{\pi}$ le caract\`ere d'Harish-Chandra de $\pi$. La \og conjecture d'Osborne non-tordue \fg,
d\'emontr\'ee par Hecht et Schmid, relie $\Theta_{\pi} (t)$ aux traces $\Theta_q (t, V)$ de 
$t\in T_H$ (r\'egulier) dans $H_q (\n_H , V)$ et au d\'enominateur 
$$D_{\n_H} (t) = \det (1-t_{|\n_H}).$$

Pour exploiter ce r\'esultat et sa version \og tordue \fg, il nous faut une relation entre $\Theta_{\pi}$ --
ou plut\^ot $\Theta_{\small\prod (\psi )} = \sum_{\pi \in \small\prod (\psi)} m(\pi ) \Theta_{\pi}$ --
et le caract\`ere tordu de $\Pi ' \otimes \Pi ''$ (pour $A_{\theta}$). 

\subsection{Correspondance de classes de conjugaison entre $H$ et $H' \times H''$} 
Les classes de conjugaison semi-simples de $H$ sont repr\'esent\'ees par 
\begin{equation} \label{eq:classconj}
\begin{split}
& t  = \mathrm{diag} (x_1 , \ldots , x_{\ell} , x_{\ell}^{-1} , \ldots , x_{1}^{-1}) \in T_H \ \ \ \mbox{ si } H= \SO(2\ell , \CC) \mbox{ ou } \Sp(2\ell , \CC) , \\
& t = \mathrm{diag} (x_1 , \ldots , x_{\ell} , 1 , x_{\ell}^{-1} , \ldots , x_{1}^{-1} ) \in T_H \ \ \ \mbox{ si } H=\SO (2\ell +1 , \CC )
\end{split}
\end{equation}
modulo le groupe de Weyl de $H$. Remarquons que le groupe de Weyl de $H$ est \'egal \`a 
$W=\mathfrak{S}_{\ell} \rtimes \{\pm 1 \}^{\ell}$ si $H = \SO (2\ell +1 , \CC)$ ou $\Sp (2\ell , \CC)$ et est \'egal \`a $\mathfrak{S}_{\ell} \rtimes \{\pm 1 \}^{\ell-1}$ si
$H=\SO(2\ell , \CC)$, o\`u $\{-1\}^{\ell -1}$ est le sous-groupe de $\{-1\}^{\ell}$ d\'efini par $\prod s_i =1$.

On repr\'esente de m\^eme les classes de conjugaison semi-simple de $H' \times H''$ par des couples
$(t' , t'')$ o\`u $t '$ (resp. $t''$) est associ\'e \`a $(y_1' , \ldots , y_{\ell '} ')$ (resp. $(y_1 '' , \ldots , y_{\ell ''}'')$ 
comme dans \eqref{eq:classconj}. On note alors $\mathcal{A}_{(H' \times H'' )/H}$ l'application qui \`a la classe de conjugaison semi-simple $(t' , t") \in H' \times H''$ associe la classe de conjugaison de $t \in H$ d\'efinie comme
dans \eqref{eq:classconj} avec $\{x_1 , \ldots , x_{\ell} \} = \{ y_1' , \ldots , y_{\ell '} ' , y_1 '' , \ldots , y_{\ell ''}''\}$. 

On pose 
\begin{equation} \label{Delta1}
\Delta_{H' \times H'', H} ((t', t''), t) = \frac{|\det (\mathrm{Ad} (t) -1)|_{\mathfrak{h} / \mathfrak{t}_H}^{\frac12}}{|\det (\mathrm{Ad} (t') -1)|_{\mathfrak{h'} / \mathfrak{t}_{H'}}^{\frac12} |\det (\mathrm{Ad} (t'') -1)|_{\mathfrak{h''} / \mathfrak{t}_{H''}}^{\frac12}},
\end{equation}
o\`u $t = \mathcal{A}_{(H' \times H'' )/H} (t',t'')$. C'est le facteur de transfert que d\'efinissent 
Langlands et Shelstad \cite{LanglandsShelstad}; seul le facteur $\Delta_{\mathrm{IV}}$ est non trivial, la
cohomologie galoisienne \'etant triviale dans le cas complexe.

\subsection{} On dispose \'egalement des applications naturelles $\mathcal{A}_{H/G}$ (resp. $\mathcal{A}_{H'/G'}$, $\mathcal{A}_{H''/G''}$) des classes de conjugaison
dans $H$ (resp. $H'$, $H''$) vers les 
classes de $\theta$-conjugaison dans $G$ (resp. $G'$, $G''$). 
Et si $\tilde{t} = \mathcal{A}_{H/G} (t)$, on pose~: 
\begin{equation} \label{Delta2}
\Delta_{H , G} (t , \tilde{t}) =  \frac{| \det (\mathrm{Ad} (\tilde{t}) \circ \theta - 1 ) |^{\frac12}_{\mathfrak{g} / \mathfrak{t}}}{|\det ( \mathrm{Ad} (t) - 1 ) |^{\frac12}_{\mathfrak{h}/ \mathfrak{t}_H}} .
\end{equation}
(C'est le facteur $\Delta_{\mathrm{IV}}$ de Kottwitz-Shelstad \cite[p. 46]{KS}.)
Noter que d'apr\`es l'exemple \ref{1.4}, $\Delta_{H,G} (t,\tilde{t})=1$ lorsque $H = \SO (2\ell+1 , \CC)$ ou 
$\Sp (2\ell , \CC)$.

Cela dit, la relation \eqref{TrIdent2} est \'equivalente \`a la 
relation suivante entre le caract\`ere tordu de $\Pi' \otimes \Pi ''$ et les caract\`eres $\Theta_{\pi}$ ($\pi 
\in \prod (\psi )$).

\begin{lem} \label{55}
Pour $(\tilde{t}' , \tilde{t}'')$ et $t$ associ\'es -- soit $t = \mathcal{A}_{(H' \times H'') / H} (t' , t'')$ avec 
$t' = \mathcal{N}_{H'/G'} (\tilde{t}')$ et $t'' = \mathcal{N}_{H''/G''} (\tilde{t}'')$ -- on a~:
$$\Theta_{\Pi ' \otimes \Pi '', \theta} (\tilde{t}', \tilde{t}'') = \varepsilon \frac{\Delta_{H' \times H'', H} ((t', t''), t)}{
\Delta_{H' , G'} (t' , \tilde{t}') \Delta_{H'' , G''} (t'' , \tilde{t}'')} \Theta_{\small\prod (\psi )} (t).$$
\end{lem}

Nous pouvons maintenant comparer les expressions donn\'ees par les versions tordues et non tordues
de la \og conjecture d'Osborne \fg.

Rappelons que $H_0 (\mathfrak{n}_H , V)$ est li\'e par la r\'eciprocit\'e de Frobenius aux homomorphismes
de $V$ vers les induites \cite[\S 4]{HechtSchmid}. On a
$${\rm Hom}_H \left( H_0 (\mathfrak{n}_H , V) , \CC_{\chi} \right) = {\rm Hom}_H \left( V , J_{\chi} \right)$$
pour $\chi$ un caract\`ere de $T_H$, $J_{\chi}$ \'etant l'induite non normalis\'ee. On en d\'eduit
$${\rm Hom}_H \left( H_0 (\mathfrak{n}_H , V) e^{-\rho_H} , \CC_{\chi} \right) = {\rm Hom}_H \left( V , I_{\chi} \right)$$
o\`u $I_{\chi}$ est l'induite {\it unitaire}. Comme on l'a vu, les calculs d'exposants s'expriment 
plus naturellement dans ce cadre, cf. \cite[p. 51]{HechtSchmid}.
Le caract\`ere de $H_q (\mathfrak{n}_H , V) e^{-\rho_H}$ est \'evidemment 
$\Theta_q ( \cdot , V ) e^{-\rho_H}$. Les
m\^emes consid\'erations s'appliquent \`a $G$. 

Les deux lemmes suivants sont laiss\'es au lecteur (pour des calculs analogues, cf. exemple \ref{1.4})~:

\begin{lem} \label{den=1}
Pour $\tilde{t}$ et $t$ r\'eguliers et associ\'es, on a~:
\begin{equation*} 
\Delta_{H,G} (t , \tilde{t})  =  \frac{e^{-\rho} (\tilde{t}) D_{\mathfrak{n}}^{\theta} (\tilde{t})}{e^{-\rho_H} (t ) D_{\mathfrak{n}_H} (t )}.
\end{equation*}
\end{lem}

\begin{lem} \label{den=2}
Pour $t$ et $(t',t'')$ r\'eguliers et associ\'es, on a~:
\begin{equation*} 
\Delta_{H' \times H'',H} ((t', t'') , t)  =  \frac{e^{-\rho_H} (t) D_{\mathfrak{n}_H} (t)}{e^{-\rho_{H'}} (t' ) D_{\mathfrak{n}_{H'}} (t' ) e^{-\rho_{H''}} (t'' ) D_{\mathfrak{n}_{H''}} (t'' )}.
\end{equation*}
\end{lem}

Soit $W'$ (resp. $W''$) le module d'Harish-Chandra de $\Pi'$ (resp. $\Pi''$). 
On d\'eduit alors des lemme \ref{55}, \ref{den=1} et \ref{den=2} et du th\'eor\`eme
\ref{Osborne} l'\'egalit\'e
\begin{multline} \label{id1}
\sum_{q' , q''} (-1)^{q'+q''} \Theta_{q'}^{\theta} (\tilde{t}' ,W' ) e^{-\rho'} (\tilde{t}')
 \Theta_{q''}^{\theta} (\tilde{t}'' , W'')e^{-\rho ''} (\tilde{t} '') \\ 
= \varepsilon \sum_{\pi \in \small\prod (\psi )}  m (\pi) \left\{ \sum_q (-1)^q \Theta_q (t , V) e^{-\rho_H} (t) \right\} .
\end{multline} 
Lorsque $H= \SO (2\ell , \CC)$ une repr\'esentation $\pi \in \small\prod (\psi)$ est consid\'er\'ee \og modulo $\alpha$ \fg, elle d\'etermine alors une paire $\{ \pi , \pi ' \}$ de vraies repr\'esentations de $G$.
Dans ce cas il faut remplacer le membre de droite de \eqref{id1} par~:
$$\varepsilon \sum_{\pi \in \small\prod (\psi )}  m (\pi) \left\{ \sum_q (-1)^q \left[ \Theta_q (t , V)  + \Theta_q (t , V') \right] e^{-\rho_H} (t) \right\},$$ 
o\`u l'on a not\'e $V$ et $V'$ les modules d'Harish-Chandra associ\'ees \`a $\pi$ et $\pi'$. 

\subsection{D\'emonstration de la proposition \ref{prop:exp}}

Soit $(\pi , V) \in \small\prod (\psi )$. Consid\'erons le terme du membre de droite de 
\eqref{id1} associ\'e \`a $H_0 (\mathfrak{n}_H , V)$. Il intervient avec un signe $\varepsilon$ qui 
ne d\'epend pas de $\pi$. Il d\'ecoule donc de \cite[Prop. 3.2]{HechtSchmid} -- c'est-\`a-dire de la proposition \ref{ExpMax} -- que tout exposant 
minimal, pour l'ordre $\leq_H$, doit appara\^{\i}tre dans le membre de gauche de \eqref{id1} 
sauf s'il est annul\'e 
par un exposant d'une autre repr\'esentation $(\pi_1 , V_1)$ (dont l'homologie appara\^{\i}t en degr\'e
impair). 

Si $e$ est un exposant minimal et intervient dans $H_0 (\mathfrak{n}_H , V)$, on voit donc que 
$e$ subsiste dans le membre de gauche, ou bien que $e=e''$ o\`u $e''$ appara\^{\i}t, en degr\'e impair,
dans $H_q (\mathfrak{n}_H , V_1 )$. Alors $e \geq_H e'$ o\`u $e'$ est un exposant minimal de $V_1$.
(De plus $e\neq_H e'$). Par r\'ecurrence on voit que $e  \geq_H e '$ o\`u (en changeant de notation) $e'$
est minimal et subsiste dans le membre de gauche de \eqref{id1}. 

On s'est donc ramen\'e au cas o\`u $e$ subsiste dans le membre de gauche de \eqref{id1}. En d'autres termes il existe  des $\theta$-exposants d'homologie $e'$ et $e''$  de $\Pi'$ et $\Pi''$ tels que $e=e'+e''$.

\section{$\Theta$-exposants des groupes lin\'eaires}

On note toujours $G = \GL (N , \CC)$ ($N \geq 1$) et $\theta$ l'automorphisme involutif 
$g \mapsto J {}^t \; g^{-1} J^{-1}$.  

Soit $\Pi$ la repr\'esentation \eqref{Pi} de $G$ et $W$ son module d'Harish-Chandra.
On note
$e_{\psi}$ l'exposant $(m_1 \leq \ldots \leq m_{\ell}) \in \CC^{\ell}$ o\`u les $m_j \in \N$ sont 
les \'el\'ements des segments $\sigma_i = \left( a_i -1, a_i -2 , \ldots  \right)$, rang\'es par ordre croissant.
On note\footnote{Les groupes $E$ des paragraphes pr\'ec\'edents n'interviennent pas dans cet argument; on esp\`ere que la notation ne pr\^ete pas \`a confusion.} $E_{\psi}$ l'exposant total $\theta$-stable $(m_1 \leq \ldots \leq m_N)$. Noter que $E_{\psi}$, compos\'e avec la norme $\mathcal{N}_{H/G}$, est \'egal \`a $e_{\psi}$ quand $H$ est un sous-groupe endoscopique avec un seul facteur. 

\subsection{} On note encore $\leq_{\theta}$ 
l'ordre induit par $\leq_{\theta}$ sur les exposants de $H$. En d'autres termes, $\leq_{\theta}$ correspond toujours \`a l'ordre donn\'e par le groupe $\SO(2\ell +1)$; celui-ci ne co\"{\i}ncide avec un
sous-groupe endoscopique $H$ que lorsque $N=2\ell$ est pair et $H=\SO(2\ell +1)$.

Noter que l'on dispose de la m\^eme mani\`ere d'ordres $\leq_{\theta}'$ et $\leq_{\theta}''$ sur les
exposants $\theta$-stables de $G'$ et $G''$. 
L'identification de $A$ au produit $A' \times A''$ permet alors de d\'efinir un 
ordre partiel $(\leq_{\theta} ' \times \leq_{\theta} '')$ sur les exposants $\theta$-stables de $G$. En d'autres termes, $(\leq_{\theta} ' \times \leq_{\theta} '')$ est l'ordre donn\'e par les sommes de racines 
$N\alpha$ des facteurs $G'$ et $G''$ dans $G$. {\it Via} les applications normes l'ordre partiel
$(\leq_{\theta} ' \times \leq_{\theta} '')$ induit un ordre sur les exposants de $H$.
On a imm\'ediatement~:
\begin{equation} \label{plusfin}
e(\leq_{\theta} ' \times \leq_{\theta} '') e' \Rightarrow e \leq_{\theta} e '.
\end{equation}

Le th\'eor\`eme \ref{thm:intro} d\'ecoule imm\'ediatement de la proposition \ref{prop:exp}, de \eqref{plusfin} et de la proposition
suivante dont la d\'emonstration fait l'objet de cette section.
 
\begin{prop} \label{Pexp}
Pour tout $\theta$-exposant d'homologie $E$ de $\Pi$, on a $\mathrm{Re} (E) \geq_{\theta} E_{\psi}$.
\end{prop}

Pour la d\'emonstration, nous oublions pour l'instant la pr\'esence de $\theta$. 
D'apr\`es la classification de Langlands (et l'absence de $L$-indiscernabilit\'e pour les groupes 
lin\'eaires ou complexes), on peut r\'ealiser $\Pi$ comme l'unique sous-module irr\'eductible de l'induite normalis\'ee
\begin{equation} \label{ind}
\mathrm{ind}_B^G ( \eta_1 \otimes \ldots  \otimes \eta _N)
\end{equation}
o\`u $B$ est le sous-groupe de Borel et $\eta_1 , \ldots , \eta_N$ sont obtenus ainsi~: On \'ecrit $a_1 \geq \ldots  \geq a_m$. Alors 
\begin{multline*}
\eta_1 = | \cdot |^{\frac{1-a_1}{2}} \chi_1 , \ \eta_2 = |\cdot |^{\frac{1-a_1}{2}} \chi_2 \ (\mbox{si } a_2 = a_1) \\ 
\ldots \ \eta_k = | \cdot |^{\frac{1-a_1}{2}} \chi_k \ (\mbox{si } a_k = a_1).
\end{multline*}
On range ensuite les caract\`eres de valeur absolue
sup\'erieure, {\it etc}. Noter que la formulation la plus r\'epandue de la classification de Langlands conduit \`a r\'ealiser $\Pi$ comme quotient; la r\'ealisation ci-dessus s'en d\'eduit par dualit\'e. Puisque $\Pi$ est $\theta$-stable, on peut en outre arranger les $\eta_i$ de sorte que le caract\`ere $\eta$ obtenu soit $\theta$-invariant. L'unicit\'e de la classification de Langlands implique alors que si 
$(a_i , \chi_i ')$ est une autre donn\'ee, on a $(a_i , \chi_i ') = (a_i , \chi_i)$ \`a permutation pr\`es. 

Noter que l'exposant $\theta$-stable dans $H_0 (\n , W)$ associ\'e \`a la r\'ealisation \eqref{ind} de 
$\Pi$ a pour partie r\'eelle $E_{\psi}$. 

\subsection{} Consid\'erons un exposant d'homologie $E$ de $\Pi$. D'apr\`es \cite[Prop. 3.2]{HechtSchmid} (ou proposition \ref{ExpMax}), 
l'exposant $E$ est sup\'erieur -- pour l'ordre de $G$ -- \`a un exposant d'homologie $E_0$ intervenant dans $H_0  (\mathfrak{n} , W)$. Et d'apr\`es le th\'eor\`eme de r\'eciprocit\'e de Frobenius, il correspond\footnote{Comme au paragraphe \ref{615b}.} \`a l'exposant $E_0$  dans
$H_0 (\mathfrak{n} , W)$ un homomorphisme non-nul $W \rightarrow I$ vers une induite 
$$I = \mathrm{ind}_B^G (\lambda_1 \otimes \ldots \otimes \lambda_N )$$
avec $\lambda_j = z^{p_j} \overline{z}^{q_j} = (z/ \overline{z})^{\mu_j} (z \overline{z})^{\nu_j}$ et
$E_0= (p_1+q_1 , \ldots , p_N + q_N)$.\footnote{On a $\nu_j = \frac12 (p_j +q_j)$, $\mu_j = \frac12 (p_j-q_j)$. En particulier $\mu_j$ est un demi-entier.} 

Si l'image de $W$ est le sous-module de Langlands $L \subset I$, il d\'ecoule de l'unicit\'e de la classification de Langlands que $E_0=E_{\psi}$. Sinon le param\`etre de $I$ n'est pas anti-dominant, donc
on peut trouver $i$ tel que $\nu_i > \nu_{i+1}$. Dans $\GL(2, \CC)$ on a un op\'erateur d'entrelacement 
non-nul 
\begin{equation} \label{entrelacement}
\mathrm{ind} (\lambda_i \otimes \lambda_{i+1}) \rightarrow \mathrm{ind} (\lambda_{i+1} \otimes \lambda_i )
\end{equation}
o\`u le caract\`ere $\lambda_i$ est de param\`etres $(\mu_i , \nu_i )$. Consid\'erant l'induite totale \`a $G$ on en d\'eduit un 
op\'erateur d'entrelacement 
\begin{equation} \label{eq:entralct}
I \rightarrow I'.
\end{equation}
Il y a deux possibilit\'es~: Si c'est un isomorphisme, on a \'evidemment un morphisme non-nul 
$W \rightarrow I'$ et on peut remplacer $(\nu_i , \nu_{i+1})$ par~:
\begin{equation} \label{eq:chgtrac1}
(\nu_i ' , \nu_{i+1} ') = (\nu_{i+1} , \nu_i ) = (\nu_i , \nu_{i+1}) - (\nu_i - \nu_{i+1}) (1, -1) .
\end{equation}
Le cas o\`u \eqref{eq:entralct} n'est pas un isomorphisme n'appara\^it que si 
$\lambda_i \lambda_{i+1}^{-1} (z) = z^p \overline{z}^q$, avec $p$ et $q$ entiers $\geq 1$, voir par exemple \cite[Thm. 6.2]{JL} ou \cite{Duflo} pour des r\'esultats plus complets sur les op\'erateurs d'entrelacement pour les groupes complexes.
Dans ce cas le noyau de \eqref{entrelacement} est une induite $\mathrm{ind} (\lambda_i ' \otimes \lambda_{i+1} ')$
avec 
$$\lambda_i ' \lambda_{i+1} ' = \lambda_i \lambda_{i+1} \mbox{ et } \lambda_i ' (\lambda_{i+1} ')^{-1} = 
z^p \overline{z}^{-q}.$$
En particulier, les $\nu_i'$ et $\nu_{i+1}'$ associ\'es aux caract\`eres $\lambda_i'$ et $\lambda_{i+1}'$ 
v\'erifient~:
\begin{equation} \label{eq:chgtrac2}
(\nu_i ' , \nu_{i+1} ') = (\nu_i , \nu_{i+1}) -q (\frac12,- \frac12) .
\end{equation}
On a alors une suite exacte de repr\'esentations de $\GL (2 , \CC)$~:
$$0 \rightarrow \mathrm{ind} (\lambda_i ' \otimes \lambda_{i+1} ') \rightarrow \mathrm{ind} (\lambda_i \otimes \lambda_{i+1} ) \rightarrow F \rightarrow 0,$$
$F$ \'etant d'ailleurs de dimension finie d'o\`u par induction
$$0 \rightarrow J \rightarrow I \rightarrow \mathrm{ind} \ F \rightarrow 0 $$
(repr\'esentations de $G$). Si l'image de $W$ dans $I$ s'envoie non trivialement dans $\mathrm{ind} \ F
\subset I'$, on est r\'eduit au cas pr\'ec\'edent. Sinon on obtient $W \rightarrow J$, application non-nulle, et on peut remplacer $(\nu_i , \nu_{i+1})$ par $(\nu_i ' , \nu_{i+1} ')$ donn\'e par \eqref{eq:chgtrac2}.

Notons que l'ordre sur les exposants de $G$ est d\'efini par les racines r\'eelles sur $A \cong (\R^{\times})^N$ (cf. \cite{HechtSchmid} ainsi que le \S \ref{S4}). En particulier les racines simples sont
donn\'ees, pour $x=(x_i) \in A$, par $\mathrm{diag} (x_i ) \mapsto x_i x_{i+1}^{-1}$. D'apr\`es 
\eqref{eq:chgtrac1} et \eqref{eq:chgtrac2} on a donc pour le nouvel exposant $E'$
$$E' = E - 2 (\nu_i - \nu_{i+1}) \alpha_i $$
ou bien 
$$E' = E -q \alpha_i.$$
Dans le premier cas, $\nu_i - \nu_{i+1} = \frac{p+q}{2}$ est un demi-entier $>0$. On voit donc que dans le cas de r\'eductibilit\'e 
$$E' \leq E.$$

Apr\`es un nombre fini de telle op\'erations on obtient 
$$0 \rightarrow W \rightarrow I^{(k)}$$
o\`u l'induite $I^{(k)}$ est en position de Langlands, et donc l'exposant associ\'e est de partie r\'eelle 
\'egale \`a $E_{\psi}$. On a donc pour tout exposant d'homologie
$$\mathrm{Re} (E) \geq E_{\psi}.$$

\subsection{D\'emonstration de la proposition \ref{Pexp}} Supposons maintenant que $E$ est $\theta$-stable. Alors 
$$\mathrm{Re}  (E) = E_{\psi} + \sum_{\alpha} m_{\alpha} \alpha,$$
o\`u $\alpha$ parcourt les racines simples de $G$. Puisque $E$ et $E_{\psi}$ sont tous les deux 
$\theta$-stables on a alors~:
$$\mathrm{Re}  (E) = E_{\psi} + \sum_{\mathrm{N}\alpha} m_{\alpha} \mathrm{N}\alpha.$$
Autrement dit~:
$$\mathrm{Re}(E) \geq_{\theta} E_{\psi}.$$

\bibliography{bibli}

\def\cftil#1{\ifmmode\setbox7\hbox{$\accent"5E#1$}\else
  \setbox7\hbox{\accent"5E#1}\penalty 10000\relax\fi\raise 1\ht7
  \hbox{\lower1.15ex\hbox to 1\wd7{\hss\accent"7E\hss}}\penalty 10000
  \hskip-1\wd7\penalty 10000\box7} \def\cprime{$'$}
\providecommand{\bysame}{\leavevmode ---\ }
\providecommand{\og}{``}
\providecommand{\fg}{''}
\providecommand{\smfandname}{\&}
\providecommand{\smfedsname}{\'eds.}
\providecommand{\smfedname}{\'ed.}
\providecommand{\smfmastersthesisname}{M\'emoire}
\providecommand{\smfphdthesisname}{Th\`ese}
\begin{thebibliography}{10}

\bibitem{Arthur}
{\scshape J.~Arthur} -- {\og An introduction to the trace formula\fg}, in
  \emph{Harmonic analysis, the trace formula, and {S}himura varieties}, Clay
  Math. Proc., vol.~4, Amer. Math. Soc., Providence, RI, 2005, p.~1--263.

\bibitem{Baruch}
{\scshape E.~M. Baruch} -- {\og A proof of {K}irillov's conjecture\fg},
  \emph{Ann. of Math. (2)} \textbf{158} (2003), no.~1, p.~207--252.

\bibitem{Bernstein}
{\scshape J.~N. Bernstein} -- {\og {$P$}-invariant distributions on {${\rm
  GL}(N)$} and the classification of unitary representations of {${\rm GL}(N)$}
  (non-{A}rchimedean case)\fg}, in \emph{Lie group representations, {II}
  ({C}ollege {P}ark, {M}d., 1982/1983)}, Lecture Notes in Math., vol. 1041,
  Springer, Berlin, 1984, p.~50--102.

\bibitem{Bouaziz}
{\scshape A.~Bouaziz} -- {\og Sur les caract\`eres des groupes de {L}ie
  r\'eductifs non connexes\fg}, \emph{J. Funct. Anal.} \textbf{70} (1987),
  no.~1, p.~1--79.

\bibitem{CasselmanOsborne}
{\scshape W.~Casselman {\normalfont \smfandname} M.~S. Osborne} -- {\og The
  {$\mathfrak{n}$}-cohomology of representations with an infinitesimal
  character\fg}, \emph{Compositio Math.} \textbf{31} (1975), no.~2,
  p.~219--227.

\bibitem{ChenevierClozel}
{\scshape G.~Chenevier {\normalfont \smfandname} L.~Clozel} -- {\og Corps de
  nombres peu ramifi\'es et formes automorphes autoduales\fg}, \emph{J. Amer.
  Math. Soc.} \textbf{22} (2009), no.~2, p.~467--519.

\bibitem{ClozelABS}
{\scshape L.~Clozel} -- {\og The {ABS} principle: consequences for
  {$L^2(G/H)$}\fg}, in \emph{On certain {$L$}-functions}, Clay Math. Proc.,
  vol.~13, Amer. Math. Soc., Providence, RI, 2011, p.~99--115.

\bibitem{Duflo}
{\scshape M.~Duflo} -- {\og Repr\'esentations irr\'eductibles des groupes
  semi-simples complexes\fg}, in \emph{Analyse harmonique sur les groupes de
  {L}ie ({S}\'em., {N}ancy-{S}trasbourg, 1973--75)}, Springer, Berlin, 1975,
  p.~26--88. Lecture Notes in Math., Vol. 497.

\bibitem{FS}
{\scshape A.~I. Fomin {\normalfont \smfandname} N.~N. {\v{S}}apovalov} -- {\og
  A certain property of the characters of irreducible representations of real
  semisimple {L}ie groups\fg}, \emph{Funkcional. Anal. i Prilo\v zen.}
  \textbf{8} (1974), no.~3, p.~87--88.

\bibitem{HC1}
{\scshape Harish-Chandra} -- {\og Invariant eigendistributions on a semisimple
  {L}ie group\fg}, \emph{Trans. Amer. Math. Soc.} \textbf{119} (1965),
  p.~457--508.

\bibitem{HechtSchmid}
{\scshape H.~Hecht {\normalfont \smfandname} W.~Schmid} -- {\og Characters,
  asymptotics and {$\mathfrak{n}$}-homology of {H}arish-{C}handra modules\fg},
  \emph{Acta Math.} \textbf{151} (1983), no.~1-2, p.~49--151.

\bibitem{Hirai}
{\scshape T.~Hirai} -- {\og The characters of some induced representations of
  semi-simple {L}ie groups\fg}, \emph{J. Math. Kyoto Univ.} \textbf{8} (1968),
  p.~313--363.

\bibitem{JL}
{\scshape H.~Jacquet {\normalfont \smfandname} R.~P. Langlands} --
  \emph{Automorphic forms on {${\rm GL}(2)$}}, Lecture Notes in Mathematics,
  Vol. 114, Springer-Verlag, Berlin, 1970.

\bibitem{Knapp}
{\scshape A.~W. Knapp} -- \emph{Representation theory of semisimple groups},
  Princeton Landmarks in Mathematics, Princeton University Press, Princeton,
  NJ, 2001, An overview based on examples, Reprint of the 1986 original.

\bibitem{KnappVogan}
{\scshape A.~W. Knapp {\normalfont \smfandname} D.~A. Vogan, Jr.} --
  \emph{Cohomological induction and unitary representations}, Princeton
  Mathematical Series, vol.~45, Princeton University Press, Princeton, NJ,
  1995.

\bibitem{KS}
{\scshape R.~E. Kottwitz {\normalfont \smfandname} D.~Shelstad} -- {\og
  Foundations of twisted endoscopy\fg}, \emph{Ast\'erisque} (1999), no.~255,
  p.~vi+190.

\bibitem{Labesse}
{\scshape J.-P. Labesse} -- {\og Stable twisted trace formula: elliptic
  terms\fg}, \emph{J. Inst. Math. Jussieu} \textbf{3} (2004), no.~4,
  p.~473--530.

\bibitem{LanglandsShelstad}
{\scshape R.~P. Langlands {\normalfont \smfandname} D.~Shelstad} -- {\og On the
  definition of transfer factors\fg}, \emph{Math. Ann.} \textbf{278} (1987),
  no.~1-4, p.~219--271.

\bibitem{Moeglin}
{\scshape C.~M{\oe}glin} -- {\og Comparaison des param\`etres de {L}anglands et
  des exposants \`a l'int\'erieur d'un paquet d'{A}rthur\fg}, \emph{J. Lie
  Theory} \textbf{19} (2009), no.~4, p.~797--840.

\bibitem{VoganU}
{\scshape D.~A. Vogan, Jr.} -- {\og The unitary dual of {${\rm GL}(n)$} over an
  {A}rchimedean field\fg}, \emph{Invent. Math.} \textbf{83} (1986), no.~3,
  p.~449--505.

\bibitem{Waldspurger}
{\scshape J.-L. Waldspurger} -- {\og Le groupe {${\bf GL}_N$} tordu, sur un
  corps {$p$}-adique. {I}\fg}, \emph{Duke Math. J.} \textbf{137} (2007), no.~2,
  p.~185--234.

\bibitem{Waldspurger2}
\bysame , {\og Les facteurs de transfert pour les groupes classiques: un
  formulaire\fg}, \emph{Manuscripta Math.} \textbf{133} (2010), no.~1-2,
  p.~41--82.

\end{thebibliography}

\bibliographystyle{smfplain}

\end{document}